\RequirePackage{fix-cm}
\documentclass[smallextended]{svjour3}
\smartqed

\usepackage[T1]{fontenc}
\usepackage[utf8]{inputenc}
\usepackage{lmodern}
\usepackage{mathptmx}
\usepackage{mathtools}
\usepackage{amsfonts}
\usepackage{bbm,bm}
\usepackage{enumitem}
\usepackage[dvipsnames]{xcolor}
\usepackage{mathrsfs}
\usepackage{upgreek}
\usepackage{aligned-overset}
\usepackage{booktabs}
\usepackage{array}
\usepackage{algorithm}
\usepackage{subfigure}
\usepackage[numbers,sort&compress]{natbib}
\usepackage[hidelinks]{hyperref}
\usepackage[nameinlink,capitalize]{cleveref}
\usepackage{placeins}

\journalname{Journal of Scientific Computing}
\graphicspath{{graphics/}}

\newcommand{\rset}{\mathbb{R}}
\newcommand{\nset}{\mathbb{N}}

\newcommand{\Id}{\mathrm{I}}

\newcommand{\rk}[1]{\operatorname{rk}\!\left(#1\right)}

\newcommand{\bracket}[1]{\left(#1\right)}

\newcommand{\cond}{\middle|}

\newcommand{\norm}[1]{\left\lVert#1\right\rVert_2}
\newcommand{\frobnorm}[1]{\left\Vert#1\right\Vert_{\mathrm{F}}}
\newcommand{\normf}[1]{\left\lVert#1\right\rVert_{\mathrm{F}}}

\def\eqsp{\,}

\newcommand{\Prob}{\mathbb{P}}  
\newcommand{\prob}[1]{\mathbb{P}\!\left(#1\right)}
\newcommand{\PE}{\mathbb{E}}

\newcommand{\indiacc}[1]{\mathbbm{1}_{\{#1\}}}

\def\rmd{\mathrm{d}}

\def\dixon{T^{\operatorname{d}}}
\def\vanilla{T^{\operatorname{v}}}
\def\counter{T^{\operatorname{cb}}}

\title{On Upper Bounds for the Matrix Spectral Norm}
\titlerunning{Upper Bounds for the Matrix Spectral Norm}

\author{A. Naumov \and M. Rakhuba \and D. Ryapolov \and S. Samsonov}
\authorrunning{A. Naumov et al.}

\date{}

\begin{document}

\maketitle

\begin{abstract}
In this paper, we study the problem of upper bounding the spectral norm of a matrix using only a small number of matrix-vector products and propose the \emph{Counterbalance} estimator to solve this problem. The
estimator does not require any prior information about the singular values of the underlying
matrix. We prove non-asymptotic bounds for its underestimation probability and compare the \emph{Counterbalance} estimator with Dixon-type and sequential Lanczos-type estimators on matrices with different singular value profiles. 
\end{abstract}

\section{Introduction and problem setting}
\label{sec:intro}
Recently, there has been a growing interest in numerical linear algebra techniques developed for the matrix-free setting \cite{halko2011structure}. In this setting, one does not have access to the entries of the matrix, but can only compute matrix-vector
products (matvecs). Indeed, in many applications the full matrix \(A \in \mathbb{R}^{m \times n}\)
is not available explicitly, whereas products \(AX\), \(X \in \mathbb{R}^n\), and
sometimes \(A^\top Y\), \(Y \in \mathbb{R}^m\), can be computed efficiently.
\par 
In this paper, we consider the problem of upper bounding the spectral norm of a matrix in the matrix-free setting. The goal is to obtain an \emph{upper bound} on $\norm{A}$ using only a small number of matvecs. This is different from the well-studied problem of obtaining lower bounds or
approximations from below, which can be done for instance using the power method. Such upper bounds
arise, for example, in sensitivity analysis  and condition number estimation \cite{higham_exp_2013, wang2015condition},
in regularizing deep neural networks by controlling layer Jacobian norms \cite{gouk2021regularisation},
and in robustness verification \cite{weng2018towards}.
\par
A simple matrix-agnostic approach to estimate $\norm{A}$ from above is to draw a random vector $X \sim \mathcal{N}(0, \Id_{n})$ and use the statistic
\begin{equation}
\label{eq:gaus_probing_est}
    T(\theta,X)=\theta\|AX\|_2\,, 
\end{equation}
as suggested in \cite{halko2011structure}. We further refer to this statistic as a \emph{Gaussian probing term}. Its theoretical properties are well understood. In particular, for a matrix $A \in \rset^{m \times n}$ and a vector $X$ with normal distribution $X \sim \mathcal{N}(0,\Id_{n})$, it holds that (see \cite[Lemma~4.1]{halko2011structure}):
\begin{equation}
\label{eq:simple_upper_bound_prob}
\Prob(\theta \norm{AX} \leq \norm{A}) \leq \sqrt{\frac{2}{\pi}} \frac{1}{\theta}
\end{equation}
for any $\theta > 0$. This bound is sharp without further assumptions
on $A$, since it is tight for matrices  with $\rk{A} = 1$. At the same time, the statistic
\(\theta \norm{AX}\) can be quite conservative for matrices with many non-negligible
singular values. Indeed,
\[
 \PE[\norm{AX}^2] = \frobnorm{A}^2\,.
\]
so when \(\|A\|_F \gg \|A\|_2\), the random quantity \(\|AX\|_2\) is typically
already much larger than \(\|A\|_2\). Indeed, using the concentration inequalities \cite{blm:2013,vershynin:2018}, one can study the concentration properties of $\norm{AX}$ around its expectation, which can be shown to be close to $\frobnorm{A}$. At the same time, quantifying this enlargement requires
spectral information such as the \emph{stable rank}
\[
\rho=\frac{\frobnorm{A}^2}{\norm{A}^2}\,,
\]
which is usually unavailable in a matrix-free setting.
\par 
This leads to the following problem. Let $T(\theta,X)$ be a randomized statistic intended to upper bound $\norm{A}$, where $\theta$ is a tuning parameter and $X$ denotes the auxiliary randomness. We want to choose $\theta$ from a prescribed
failure probability $\delta \in (0,1)$, without using structural information about $A$, so that
\[
    \sup_{A \in \rset^{m \times n}}
    \Prob(T(\theta,X)\leq \norm{A})
    \leq \delta\,.
\]
At the same time, this worst-case choice of parameter $\theta$ should not make the resulting upper
bound unnecessarily loose. Thus, the goal is to construct a statistic that is both agnostic to the structure of the underlying matrix $A$ in its calibration (that is, the choice of prescribed parameter $\theta$) and tighter in practice than the basic Gaussian probing estimator \eqref{eq:gaus_probing_est} and its straightforward modifications, see detailed review in \Cref{sec:lit_review}. Our main contributions are as follows:
\begin{itemize}[noitemsep,topsep=0pt,leftmargin=3em]
\setlength{\itemsep}{0pt}
  \setlength{\parskip}{0pt}
  \setlength{\topsep}{2pt}
    \item We propose the Counterbalance estimator, a three-matvec randomized estimator
    with two sequential and one parallel matrix-vector products. The estimator combines
    a power-type term with an independent Gaussian probing term. We derive a
    non-asymptotic, matrix-agnostic calibration rule for its parameter,
    guaranteeing a prescribed underestimation probability without using spectral properties of \(A\). 
     \item We provide a new bound (Theorem~\ref{Theorem 1}) on the underestimation probability of the classical two-step power method, demonstrating that its efficiency as an upper-bound estimator increases as the stable rank $\rho$ decreases.
     This analysis highlights the limitations of power-type estimators and motivates the proposed \emph{Counterbalance} approach.

    \item We provide numerical experiments illustrating the behavior of the estimator on
    matrices with different stable ranks and singular-value profiles. In the tested low-budget regimes, \emph{Counterbalance}
    gives smaller empirical overestimation and $\operatorname{MAE}$ than the considered Gaussian probe and
    Dixon-type baselines at the same target failure probability.
\end{itemize}

\paragraph{Notations} For a matrix $A \in \rset^{m \times n}$, we write $\norm{A}$ for its spectral norm and $\frobnorm{A}$ for its Frobenius norm, $\rk{A}$ for its rank. We denote by $\rho$ its \emph{stable rank} $\rho = \frobnorm{A}^2 / \norm{A}^2$. Given that the matrix $A$ has rank $\rk{A} = r$ admits a full singular value decomposition (SVD) in the form $A = U \Sigma V^{\top}$, $U \in \mathbb{R}^{m \times m}, V \in \mathbb{R}^{n \times n} $, where $\Sigma$ is diagonal with elements $\sigma_1 \geq \sigma_2 \geq \ldots \geq \sigma_r > 0$, which are the corresponding singular values, stable rank of $A$ can be written as $\rho = \frobnorm{A}^2 / \norm{A}^2 = (\sum_{i=1}^r\sigma_i^2)/\sigma_1^2$. We also denote by $\chi_{(1, \alpha)}^2(t) = \prob{\xi^2 + \alpha \eta^2 \le t}$, where $\xi, \eta$ are i.i.d. random variables with standard normal distribution $\mathcal{N}(0, 1)$, and write $p_{(1, \alpha)}(t)$ for the corresponding density function. We write $\chi^2_{k}(t)$, $k \in \nset$, for the distribution function of the chi-squared distribution with $k$ degrees of freedom.
In particular, we write $\chi^2_{1}(t)$ instead of $\chi^2_{(1,0)}(t)$ for conciseness. For a statistic $T(\theta, X)$ that aims to upper bound $\norm{A}$, we refer to the quantity $\prob{T(\theta, X) \le \norm{A}}$ as the underestimation probability.

\section{Literature review}
\label{sec:lit_review}

Randomized matrix-vector probing is a classical approach for estimating spectral quantities of
matrices that are available only through matrix-vector products. Early randomized estimates were
suggested in \cite{dixon1983}, where test vectors were sampled uniformly from the unit sphere.
This line of work was later generalized to Gaussian and Rademacher test vectors in the randomized
numerical linear algebra literature; see, for instance, \cite{halko2011structure}. A particularly
simple consequence is that, for \(X\sim N(0,I_n)\), the statistic \(\theta\norm{AX}\) yields a
matrix-independent probabilistic upper bound on \(\norm{A}\). This feature makes the estimator
attractive in black-box settings, since its calibration does not require any information about the
singular values of \(A\). Multiple independent probes can be combined by taking the maximum over
the sampled values, which decreases the underestimation probability and is fully parallelizable. In particular, for samples $X_1,\dots,X_k$ it holds that
\begin{equation}
\label{eq:maximum}
\prob{\theta\max_{1\le i\le k}\norm{AX_i} \le \norm{A}} \le \left(\sqrt{\frac{2}{\pi}} \frac{1}{\theta}\right)^k \eqsp.
\end{equation}
This multi-sample approach not only reduces the risk of a ''bad'' draw that misses the top singular vector, but also can be straightforwardly parallelized. However, the corresponding worst-case calibration is governed by rank-one matrices. As a result, it can be overly conservative for matrices with large stable rank. We use this multi-sample Gaussian probe as one of the baseline estimators in the numerical experiments below.
\par 
Another important class of methods is based on power and Krylov subspace iterations with a
random start. \cite{kuczynski1992} studied the power method and the Lanczos algorithm initiated
by a random vector, including probabilistic bounds on the convergence of the largest eigenvalue
approximation. However, the proposed bounds depend on either the matrix dimension or spectral gap, which is unknown in advance in practice. For this reason, we use their work
mainly as theoretical motivation for the power-type term.

Building on Krylov subspace ideas, \cite{hochstenbach2013probabilistic} developed probabilistic upper bounds
for the matrix two-norm as a post-processing step of Lanczos bidiagonalization. These bounds
exploit Lanczos and Ritz polynomials that are generated adaptively during the bidiagonalization
process. In practice, the method is most effective when several sequential Lanczos bidiagonalization
steps are available. The resulting upper bound requires numerical evaluation of a constructed polynomial root. This evaluation requires localizing this root, which in turn requires an upper bound on
\(\|A\|_F\), which may itself be unavailable without additional work. We provide comparison with this method under a fixed budget of sequential matvecs, see \Cref{sec:numerics}.


Structured random probes have also been considered for norm and trace estimation
problems. In particular, \cite{kressner2021norm} studied norm and trace estimation with random
rank-one vectors, which is especially relevant when the underlying matrix has tensor or Kronecker
structure. Such settings arise, for example, in matrix function computations and Fr\'echet derivative estimation. In our paper we do not impose additional structural assumptions on the underlying matrix and use standard Gaussian test vectors. 
\par 
Finally, there is a large body of work on randomized trace and Frobenius norm estimation.
Hutchinson's estimator \cite{hutchinson1989} and later developments
\cite{avron2011,roosta2015,meyer2021} provide efficient randomized estimates for quantities such
as \(\operatorname{tr}(A)\), \(\operatorname{tr}(A^\top A)\), and \(\|A\|_F\). These methods are
related in that they also rely on randomized matrix-vector products and concentration phenomena.
However, they target different  quantities and it is not obvious if these methods can be directly applied to our problem. 

\section{Counterbalance estimator}
\label{sec:algorithm}
The limitation of the bound \eqref{eq:simple_upper_bound_prob} is that it becomes tight for matrices of rank one, but overly conservative for matrices with large stable rank. In such cases, the required value of $\theta$ entails significant overestimation of the true norm, resulting in a loose upper bound in practice. This motivates an additional term that improves estimate's behavior for low-rank matrices.
\paragraph{Probabilistic bounds for power method} 
It is known that for rank-one matrix $A$, the iterates of power method \cite[Chapter 6]{halko2011structure} applied to symmetric matrix $A$ will converge in $1$ step to the true value $\norm{A}$. Consider the ratio \begin{equation}
\label{eq:frac_matvec_def}
\frac{\norm{A^\top A Y}}{\norm{A Y}}\eqsp, \quad Y \sim \mathcal{N}(0,\Id_n),
\end{equation}
which is motivated by randomized power method \cite{martinsson2020randomized} with 2 steps and adapted to non-symmetric case. Theorem below illustrates that power-type estimators yield values increasingly close to $\norm{A}$ as the stable rank $\rho$ approaches one.
\begin{theorem} 
\label{Theorem 1}
    Let $A \in \rset^{m \times n}$ be a nonzero matrix and $Y \sim \mathcal{N}(0, \Id_n)$. For any $t \in [0, 1)$ it holds that 
    \begin{align*}
        \prob{ \dfrac{\norm{A^\top AY}}{\norm{AY}}  \le \sqrt{t} \norm{A}} \le \chi_1^2 \bracket{\dfrac{(\rho - 1) t}  {1 - t }}.
    \end{align*}
\end{theorem}
 \begin{proof} First, recall singular decomposition $A = U \Sigma V^{\top}$, $U \in \mathbb{R}^{m \times m}, V \in \mathbb{R}^{n \times n} $, where elements of $\Sigma$ main diagonal, $\sigma_1 \geq \sigma_2 \geq \ldots \geq \sigma_r > 0$, are the corresponding singular values, then $\norm{AY}^2 = \sum_{i=1}^r \sigma_i^2 \xi_i^2$, where $\xi = V^\top Y \sim \mathcal{N}(0, I_n)$. Then
        \begin{align}
            &\prob{\dfrac{\norm{A^\top AY}}{\norm{AY}} \le \sqrt{t} \norm{A}} 
            = \prob{\dfrac{\sum_{i=1}^r \sigma_i^4 \xi_i^2}{\sum_{i=1}^r \sigma_i^2 \xi_i^2} \le \sigma_1^2 t} \nonumber \\
            &= \prob{\sigma_1^4 \bracket{1 - t}\xi_1^2 \le \sum_{i=2}^r \xi_i^2 \bracket{\sigma_1^2 t \sigma_i^2  - \sigma_i^4}} \leq \prob{\xi_1^2 \leq (1-t)^{-1} \sum_{i=2}^r \xi_i^2 t\dfrac{\sigma_i^2}{\sigma_1^2} } \label{eq:cdf_bound}\eqsp.
        \end{align}
        It is easy to check that $\chi_1^2(u) = \prob{\xi_1^2 \leq u}$ is concave on $\mathbb{R}_+$. Define a random variable $\eta$ by the relation
        $$
        \eta = (1-t)^{-1}\sum_{i=2}^r \xi_i^2 t\dfrac{\sigma_i^2}{\sigma_1^2}\eqsp.
        $$ 
        Note that $\eta$ is non-negative and independent of $\xi_1^2$. Moreover, 
        \[
        \PE[\eta] = \frac{(\rho - 1) t}{1 - t}\eqsp.
        \]
        Using Jensen's inequality we obtain  
        \begin{equation}
        \label{eq:jensen_bound}
        \begin{split}
        \prob{\xi_1^2 \leq \eta} &= \PE\bigl[\mathbb{I}\bracket{\xi_1^2 \le \eta}\bigr] = \PE_{\eta}\bigl[\PE_{\xi_1^2}\bigl[\mathbb{I}(\xi_1^2 \leq \eta) |  \eta \bigr] \bigr] \\
        &= \PE_{\eta}\bigl[\chi_1^2(\eta)\bigr] \leq \chi_1^2(\PE[\eta]).
        \end{split}
        \end{equation}
        Combining \eqref{eq:cdf_bound} and \eqref{eq:jensen_bound}, we get
        \begin{align*}
            \prob{\xi_1^2 \leq (1-t)^{-1} \sum_{i=2}^r \xi_i^2 \dfrac{\sigma_i^2}{\sigma_1^2}} \leq \chi_1^2\bracket{\frac{(\rho - 1) t}{1 - t}}\eqsp,
        \end{align*}
        and the statement follows.
\end{proof}

This result also illustrates that the upper bound for the \emph{stable rank} of $A$ allows for achieving given underestimation probability. Then under assumption $\theta \ge 1$ we represent the above inequality as
\begin{align*}
     \prob{ \theta \dfrac{\norm{A^\top AY}}{\norm{AY}}  \le  \norm{A}} \le \chi_1^2 \bracket{\dfrac{(\rho - 1)}  {\theta^{2} - 1 }}.
\end{align*}
In order to apply this bound in practice without additional information on $\rho$ one have to use worst-case estimate $\rho = n$. For large $n$ this estimator itself becomes loose, see \Cref{fig:power_comparison} in the experimental section. Thus, dependence of power method estimate is opposite to $\theta \norm{AX}$ in terms of $\rho$. To mitigate the weaknesses of both estimators, we propose a hybrid statistic that combines the strengths of two counter-directional terms.
 \paragraph{Our method} Given two independent vectors $X_1, X_2 \sim \mathcal{N}(0, \Id_n)$, we define:
\begin{equation}
\label{eq:conterbalance_def}
\counter(\theta,X) = \theta \sqrt{\left(\dfrac{\norm{A^\top A X_1}}{\norm{A X_1}}\right)^2 + \norm{A X_2}^2} \eqsp, \quad \text{ where } X = (X_1,X_2)\eqsp.
\end{equation}
This is the \emph{Counterbalance} estimator. The first term improves performance on low-rank matrices, while the second controls behavior when $A$ has large stable rank. The combination balances both regimes and removes the need for spectrum-aware tuning.

\par Below we provide theoretical bounds, which allow to set the parameter $\theta$ in \eqref{eq:conterbalance_def} for desired underestimation probability. Note  that the suggested statistic $\counter(\theta,X)$ relies on two independent random vectors $X_1$ and $X_2$.

\subsection{Choosing $\theta$ in the Counterbalance algorithm}
\label{sec:parameter_choice}
We now derive a matrix-agnostic bound for the underestimation probability of
\eqref{eq:conterbalance_def}.  The proof keeps one additional spectral
quantity until the final maximization.  This couples the power-type and
Gaussian-probe terms and avoids switching between two conditional bounds at a
random threshold.

Define
\begin{equation}
\label{eq:alpha_beta_def}
    a_i=\frac{\sigma_i^2}{\sigma_1^2},\qquad
    \alpha=\sum_{i=2}^r a_i=\rho-1,\qquad
    \beta=\sum_{i=2}^r a_i^2.
\end{equation}
If $\alpha>0$, then
\begin{equation}
\label{eq:beta_range}
    0<\beta\leq \alpha\min\{\alpha,1\}.
\end{equation}
Indeed, $a_i\leq1$ and $a_i\leq\alpha$, so
$a_i^2\leq a_i\min\{\alpha,1\}$.

Let $\eta_i$ and $\xi_i$ be two independent families of i.i.d. standard
normal random variables and set
\begin{equation}
\label{eq:Q_R_def}
    Q=\frac{\eta_1^2+\sum_{i=2}^r a_i^2\eta_i^2}
    {\eta_1^2+\sum_{i=2}^r a_i\eta_i^2},
    \qquad
    R=\sum_{i=2}^r a_i\xi_i^2,
    \qquad
    Z=\xi_1^2.
\end{equation}
After dividing the squared Counterbalance statistic by $\sigma_1^2$, its
underestimation event becomes
\begin{equation}
\label{eq:underestimation_QZR}
    \{\counter(\theta,X)\leq\norm{A}\}
    =\{Q+Z+R\leq x\},
    \qquad x=\theta^{-2}.
\end{equation}
Thus $Q$ describes the power-type term, while $Z+R$ is the independent
Gaussian-probe term.

Our proof relies heavily on the following property of weighted chi-squared distribution, which is due to ~\citep{weightedchi2extremum}:
\begin{theorem}[Theorem 2 from \citep{weightedchi2extremum}]
    Let $n \in \mathbb{N}, \xi_1, \dots, \xi_n$ be i.i.d random variables with distribution $\mathcal{N}(0, 1)$, and set $Q_n = \sum_{i=1}^n \lambda_i\xi_i^2$, where $\lambda_i > 0$ for any i and $\sum_{i=1}^n \lambda_i = 1$. Then, for any $x \in [0, 1]$,
    \begin{align}
        \label{eq:weighted_chi_fact}
        \sup_{\lambda_1, \dots, \lambda_n: \lambda_i \ge 0, \sum_{i=1}^n \lambda_i = 1}\prob{Q_n \le x} = \chi_1^2(x).
    \end{align}
    \end{theorem}

We now introduce two deterministic lower-tail envelopes.  The first one will
bound the CDF of $Q$.  Put $M=\min\{\alpha,1\}$ and, for $0\leq t<1$, define
\begin{equation}
\label{eq:G_alpha_beta}
    G_{\alpha,\beta}(t)
    =\chi_1^2\left(
    \frac{\alpha t-\beta\min\{t/M,1\}}{1-t}
    \right).
\end{equation}
We set $G_{\alpha,\beta}(t)=0$ for $t<0$.  The function
$G_{\alpha,\beta}$ is nondecreasing and absolutely continuous on every compact
subinterval of $[0,1)$; below $G_{\alpha,\beta}'$ denotes its almost-everywhere
derivative.

The second envelope will bound the CDF of $R$.  For $v\in\rset$, define
\begin{equation}
\label{eq:H_alpha_beta}
    H_{\alpha,\beta}(v)=
    \begin{cases}
    0, & v\leq0,\\[2pt]
    \displaystyle
    \min\left\{
    \chi_1^2\left(\frac{v}{\alpha}\right),
    \exp\left[-\frac{\alpha^2}{2\beta}
    \left(\log\frac{\alpha}{v}-1+\frac{v}{\alpha}\right)\right]
    \right\}, & 0<v<\alpha,\\[12pt]
    1, & v\geq\alpha.
    \end{cases}
\end{equation}
This is only a deterministic function used as a pointwise CDF bound.

For $\theta>1$, put $x=\theta^{-2}$ and define
\begin{equation}
\label{eq:B_theta_alpha_beta}
    B_\theta(\alpha,\beta)
    =\int_0^x p_{(1,0)}(z)
    \int_0^{x-z}
    H_{\alpha,\beta}(v)\,
    G_{\alpha,\beta}'(x-z-v)
    \,\rmd v\,\rmd z.
\end{equation}

\begin{theorem}
\label{th:main}
Let $A\in\rset^{m\times n}$ be nonzero, let $X_1,X_2$ be independent
$\mathcal{N}(0,\Id_n)$ random vectors, and let $X=(X_1,X_2)$.  For
$\theta>1$, let $\alpha=\rho-1$ and let $\beta$ be defined by
\eqref{eq:alpha_beta_def}.  If $1<\rho<7$, then
\begin{equation}
\label{eq:beta_dependent_main_bound}
    \prob{\counter(\theta,X)\leq\norm{A}}
    \leq B_\theta(\alpha,\beta).
\end{equation}
Define
\begin{equation}
\label{counterbalance_bound}
    g(\theta,\rho)=
    \begin{cases}
    0, & \rho=1,\\[2pt]
    \displaystyle
    \sup_{0<\beta\leq(\rho-1)\min\{\rho-1,1\}}
    B_\theta(\rho-1,\beta), & 1<\rho<7,\\[10pt]
    \chi_4^2(\theta^{-2}), & \rho\geq7.
    \end{cases}
\end{equation}
Then
\begin{equation}
\label{eq:main_bound}
    \prob{\counter(\theta,X)\leq\norm{A}}
    \leq g(\theta,\rho).
\end{equation}
\end{theorem}

\begin{proof}
The proof has three steps.  First we bound the lower tail of $Q$, then the
lower tail of $R$, and finally combine the two bounds using
\eqref{eq:underestimation_QZR}.

\paragraph{Step 1: upper bounding CDF of $Q$.}
\Cref{Theorem 1} is sharp in terms of $\rho$ when dimension is unrestricted, see remark \ref{remark:power_sharp_bound} for details. Hence, to obtain more accurate estimate on power term we derive upper bound in terms of $\beta$.

For $t\in[0,1)$, the event $\{Q\leq t\}$ is equivalent to
\begin{align*}
    (1-t)\eta_1^2
    \leq\sum_{i=2}^r a_i(t-a_i)\eta_i^2.
\end{align*}
Removing the negative terms on the right enlarges the event.  Conditioning on
$\eta_2,\ldots,\eta_r$ and using concavity of $\chi_1^2$ on $\rset_+$, exactly as
in the proof of \Cref{Theorem 1}, gives
\begin{equation}
\label{eq:q_intermediate_bound}
    \prob{Q\leq t}
    \leq
    \chi_1^2\left(
    \frac{\sum_{i=2}^r a_i(t-a_i)_+}{1-t}
    \right).
\end{equation}
The remaining sum satisfies the termwise inequality
\begin{equation}
\label{eq:key_termwise_bound}
    a_i(t-a_i)_+
    \leq t a_i-a_i^2\min\{t/M,1\}.
\end{equation}
Indeed, if $t\geq M$, then $a_i\leq t$ and equality holds.  If $t<M$ and
$a_i>t$, the left-hand side is zero and the right-hand side is nonnegative.
If $a_i\leq t<M$, the difference between the right- and left-hand sides is
$a_i^2(1-t/M)\geq0$.  Summing \eqref{eq:key_termwise_bound} and substituting
into \eqref{eq:q_intermediate_bound} yields
\begin{equation}
\label{eq:q_beta_cdf_bound}
    \prob{Q\leq t}\leq G_{\alpha,\beta}(t).
\end{equation}
For $t<M$, the argument of $\chi_1^2$ in \eqref{eq:G_alpha_beta} equals
$(\alpha-\beta/M)t/(1-t)$.  For $t>M$ it equals
$(\alpha t-\beta)/(1-t)$; this second case can occur only when $\alpha<1$.
The corresponding derivatives have numerators $\alpha-\beta/M$ and
$\alpha-\beta$, respectively, and are nonnegative by
\eqref{eq:beta_range}.  Hence $G_{\alpha,\beta}$ is nondecreasing and
absolutely continuous on every compact subinterval of $[0,1)$.

\paragraph{Step 2: upper bounding CDF of $R$.}
For $0\leq v\leq\alpha$, applying \eqref{eq:weighted_chi_fact} to
$R/\alpha$ gives
\begin{equation}
\label{eq:R_weighted_chi_bound}
    \prob{R\leq v}\leq \chi_1^2\left(\frac{v}{\alpha}\right).
\end{equation}
A second bound follows from the Laplace transform.  For $s>0$,
\begin{equation}
\label{eq:R_laplace_exact}
    \PE\exp\{-sR\}=\prod_{i=2}^r(1+2sa_i)^{-1/2}.
\end{equation}
The function
\begin{align*}
    h_s(u)=\frac{\log(1+2su)}{u}
    =2s\int_0^1\frac{\rmd y}{1+2suy},\qquad u>0,
\end{align*}
is convex, because the integrand is convex in $u$.  Jensen's inequality with
weights $a_i/\alpha$ therefore gives
\begin{align*}
    \sum_{i=2}^r\log(1+2sa_i)
    &=\sum_{i=2}^r a_i h_s(a_i)\\
    &\geq \alpha h_s\left(\frac{\beta}{\alpha}\right)\\
    &=\frac{\alpha^2}{\beta}
    \log\left(1+\frac{2s\beta}{\alpha}\right).
\end{align*}
Together with \eqref{eq:R_laplace_exact},
\begin{equation}
\label{eq:R_laplace_bound}
    \PE\exp\{-sR\}
    \leq
    \left(1+\frac{2s\beta}{\alpha}\right)^{-\alpha^2/(2\beta)}.
\end{equation}
For $0<v<\alpha$, the exponential Markov inequality and optimization over
$s>0$ now yield
\begin{align*}
    \prob{R\leq v}
    &\leq
    \inf_{s>0}
    \exp\{sv\}
    \left(1+\frac{2s\beta}{\alpha}\right)^{-\alpha^2/(2\beta)}\\
    &=
    \exp\left[-\frac{\alpha^2}{2\beta}
    \left(\log\frac{\alpha}{v}-1+\frac{v}{\alpha}\right)\right].
\end{align*}
Combining this estimate with \eqref{eq:R_weighted_chi_bound}, and using the
trivial bound $1$ for $v\geq\alpha$, gives
\begin{equation}
\label{eq:R_H_bound}
    \prob{R\leq v}\leq H_{\alpha,\beta}(v),
    \qquad v\in\rset.
\end{equation}

\paragraph{Step 3: combine the bounds.}
Assume $1<\rho<7$.  From \eqref{eq:underestimation_QZR}, independence of $Q$
from $(Z,R)$, and \eqref{eq:q_beta_cdf_bound},
\begin{align}
    \prob{Q+Z+R\leq x}
    &=\PE\left[\prob{Q\leq x-Z-R\cond Z,R}\right] \notag\\
    &\leq\PE\left[G_{\alpha,\beta}(x-Z-R)\right].
    \label{eq:condition_on_ZR}
\end{align}
Fix $z<x$ and write $y=x-z$.  Since $G_{\alpha,\beta}(0)=0$ and
$G_{\alpha,\beta}$ is absolutely continuous,
\begin{align*}
    G_{\alpha,\beta}(y-R)
    =\int_0^y
    \indiacc{R\leq y-t}\,
    G_{\alpha,\beta}'(t)\,\rmd t,
\end{align*}
where both sides are zero when $R>y$.  Taking expectation with respect to $R$
and applying Fubini's theorem gives
\begin{align*}
    \PE_R\left[G_{\alpha,\beta}(y-R)\right]
    &=\int_0^y
    \prob{R\leq y-t}\,
    G_{\alpha,\beta}'(t)\,\rmd t\\
    &\leq\int_0^y
    H_{\alpha,\beta}(y-t)\,
    G_{\alpha,\beta}'(t)\,\rmd t\\
    &=\int_0^y
    H_{\alpha,\beta}(v)\,
    G_{\alpha,\beta}'(y-v)\,\rmd v,
\end{align*}
where the inequality uses \eqref{eq:R_H_bound}.  Finally, integrate over
$Z\sim\chi_1^2$.  Values $z\geq x$ contribute zero, and therefore
\begin{align*}
    \prob{\counter(\theta,X)\leq\norm{A}}
    &\leq\int_0^x p_{(1,0)}(z)
    \int_0^{x-z}
    H_{\alpha,\beta}(v)
    G_{\alpha,\beta}'(x-z-v)
    \,\rmd v\,\rmd z\\
    &=B_\theta(\alpha,\beta).
\end{align*}
This proves \eqref{eq:beta_dependent_main_bound}.  Maximizing over the
admissible range \eqref{eq:beta_range} gives the middle branch of
\eqref{counterbalance_bound}.

If $\rho=1$, then $A$ has rank one and $Q=1$ almost surely.  Since
$x=\theta^{-2}<1$, the event in \eqref{eq:underestimation_QZR} is empty.

For $\rho\geq7$, the large-stable-rank argument is unchanged.  Since
$\sum_{i=2}^r a_i\geq6$ and $0<a_i\leq1$, one can greedily choose three
pairwise disjoint index sets $I_1,I_2,I_3$ such that
\begin{align*}
    1\leq\sum_{j\in I_k}a_j\leq2,
    \qquad k=1,2,3.
\end{align*}
Let $\varkappa_k=\sum_{j\in I_k}a_j\xi_j^2$.  By
\eqref{eq:weighted_chi_fact},
$\prob{\varkappa_k\leq u}\leq \chi_1^2(u)$ for every $0\leq u\leq x$.
Keeping only the Gaussian-probe term in the estimator and conditioning
successively on the other nonnegative summands therefore gives
\begin{align*}
    \prob{\counter(\theta,X)\leq\norm{A}}
    &\leq
    \prob{\xi_1^2+\varkappa_1+\varkappa_2+\varkappa_3\leq x}\\
    &\leq
    \prob{\gamma_0^2+\gamma_1^2+\gamma_2^2+\gamma_3^2\leq x}\\
    &=\chi_4^2(x),
\end{align*}
where $\gamma_0,\ldots,\gamma_3$ are independent standard normal random
variables.  Only thresholds below $x=\theta^{-2}<1$ occur in these successive
conditionings, which is exactly the range covered by
\eqref{eq:weighted_chi_fact}.  This proves the last branch of
\eqref{counterbalance_bound}.
\end{proof}

\begin{remark}
\label{remark:power_sharp_bound}
For fixed $\alpha > 0$ and take 
\[
A_d = \operatorname{diag}\left( 1, \sqrt{\alpha/d}, \dots, \sqrt{\alpha/d} \right), \quad d \ge \alpha
\]
Then $\rho = 1 + \alpha$ and 
\[
Q_d = \frac{E + (\alpha/d)^2 U_d}{E + (\alpha/d) U_d}, \quad E \sim \chi_1^2, U_d \sim \chi_d^2.
\]
Since $U_d / d \xrightarrow[]{\Prob}1$ it holds that 
\[
Q_d \xrightarrow[]{\Prob} \frac{E}{E+\alpha}
\]
Hence, for $0 < t < 1$, 
\[
\lim_{d \xrightarrow[]{} \infty} \prob{Q_d \le t} = F_1\left(\frac{\alpha_t}{1-t}\right)
\]
This establishes that \cref{Theorem 1} can not be improved uniformly over dimensions using $\rho$ alone. This motivates retaining the second spectral moment  $\beta$ in coupled analysis.
\end{remark}

\section{Practical usage of Theorem~\ref{th:main}}
\label{sec:practical}
The bound \eqref{eq:main_bound} depends on $\rho$, whereas $\rho$ is not
available in the matrix-free setting.  For $1<\rho<7$, the additional
parameter $\beta$ has already been maximized out in the definition of
$g(\theta,\rho)$.  The remaining optimization is performed offline over
$\rho\in(1,7)$; for $\rho\geq7$ the bound is already independent of the
matrix.

\begin{corollary}
\label{cor:theta_choice}
For every $\theta>1$, define
\begin{equation}
\label{eq:g_theta_matrix_agnostic}
    g(\theta)
    =\max\left\{
    \chi_4^2(\theta^{-2}),
    \sup_{1<\rho<7}g(\theta,\rho)
    \right\}.
\end{equation}
Then
\begin{equation}
\label{eq:matrix_agnostic_bound}
    \sup_{A\neq0}
    \prob{\counter(\theta,X)\leq\norm{A}}
    \leq g(\theta).
\end{equation}
Consequently, a value of $\theta$ for a prescribed failure probability
$\delta\in(0,1)$ can be selected completely offline by requiring
$g(\theta)\leq\delta$.
\end{corollary}

Numerical maximization of the coupled bound gives the values shown in
\Cref{tab:coupled_bound_check}.  The maximization is performed over
$\rho\in(1,7)$ and over the safe interval
$0<\beta\leq(\rho-1)\min\{\rho-1,1\}$; the $\rho\geq7$ branch is evaluated
separately using $\chi_4^2(\theta^{-2})$.  In particular, all values of
$\theta_{\mathrm{cb}}$ used below remain below the corresponding target
underestimation probability.

\begin{table}[ht]
\centering
\begin{tabular}{|c|c|c|c|}
\hline
$\delta$ & $\theta_{\mathrm{cb}}$ & $g(\theta_{\mathrm{cb}})$ & maximizing $\rho$ \\
\hline
$0.1$   & $1.40$ & $0.094291$  & $1.2391$ \\
$0.05$  & $1.76$ & $0.045134$  & $1.1631$ \\
$0.01$  & $3.00$ & $0.008700$  & $1.0577$ \\
$0.001$ & $6.45$ & $0.000861$ & $1.0126$ \\
\hline
\end{tabular}
\caption{Numerical maximization of the matrix-agnostic coupled bound
$g(\theta)$ at the parameter values used in the experiments.}
\label{tab:coupled_bound_check}
\end{table}

The corresponding values of $\theta$ for the Counterbalance and Vanilla
estimators are summarized in \Cref{tab:max_norm_upper_bound}.  For the Vanilla
estimator, the values are computed using three independent Gaussian probes and
the bound in \eqref{eq:maximum}.

\begin{table}[ht]
\centering
\begin{tabular}{|c|c|c|}
\hline
\textbf{$\delta$} & \textbf{$\theta_{\mathrm{cb}}$} & \textbf{$\theta_{\mathrm{v}}$} \\
\hline
0.1   & 1.4 & 1.72 \\
0.05  & 1.76 & 2.17 \\
0.01  & 3.00 & 3.70 \\
0.001 & 6.45 & 7.98 \\
\hline
\end{tabular}
\caption{Values of parameter $\theta$ for Counterbalance
\eqref{eq:conterbalance_def} and Vanilla \eqref{eq:maximum} estimators at
different $\delta$.}
\label{tab:max_norm_upper_bound}
\end{table}

\begin{algorithm}[t]
\caption{Randomized Spectral Norm Estimation}
\label{algo:operator_norm_estimator}

\textbf{Input:}
Implicit matrix $A \in \mathbb{R}^{m \times n}$,
underestimation probability $\delta \in (0,1)$,
number of matvecs $k$.

\textbf{Output:}
An upper bound on $\|A\|_2$.

\begin{enumerate}
\item Generate $X \in \rset^{n \times k}$ with columns
$X_1,\dots,X_k$ sampled i.i.d.\ from a given distribution on $\rset^n$.
\item Choose $\theta$ such that $g(\theta) \le \delta$, where
$
\Prob\bigl(T(\theta,X)\le \|A\|_2\bigr)\le g(\theta).
$
\item Compute the statistic $T(\theta,X)$.
\item \textbf{Return} $T(\theta,X)$.
\end{enumerate}

\end{algorithm}

\section{Methodology for comparing upper bounds on $\norm{A}$}
\label{sec:methodology}
In this section, we describe how to compare different randomized upper bounds on the matrix spectral norm and suggest a way to characterize their ``tightness''. Consider two randomized upper bounds $T_1(\theta,Y_1)$ and $T_2(\theta,Y_2)$ for the operator norm $\norm{A}$, each one depending on a single scalar parameter $\theta$. Here we assume that the estimators $T_1(\theta,Y_1)$ and $T_2(\theta,Y_2)$ rely on their own sets of random variables, denoted by $Y_1$ and $Y_2$, respectively. These estimators may differ in the number of random vectors used for matvec computations. We only require that both $T_1(\theta,Y_1)$ and $T_2(\theta,Y_2)$ depend on a scalar parameter $\theta > 0$ and are monotonically increasing in $\theta$ with fixed $Y_1$ and $Y_2$, respectively.
\par 


We rely on upper bounds on the underestimation probabilities to define our estimators. Formally, we assume that for any $\theta > 0$,
\begin{equation}
\label{eq:g1_g2_def}
\begin{split}
\Prob(T_i(\theta, Y_i) \leq \|A\|_2) &\leq g_i(\theta)\eqsp, \ i \in \{1, 2 \},
\end{split}
\end{equation}
where $g_i(\theta)$ are known decreasing functions. We then define the parameters $\theta_i(\delta)$:
\begin{equation}
\label{eq:theta_1_2_realizable_def}
\begin{split}
\theta_i(\delta) &= \inf \{ \theta : g_i(\theta) \leq \delta \}\eqsp.
\end{split}
\end{equation}
Given \eqref{eq:g1_g2_def}, these parameters ensure that both $T_1(\theta_1,Y_1)$ and $T_2(\theta_2,Y_2)$ have an underestimation probability at most $\delta$. Now we compare the distributions of $T(\theta_1, Y_1)$ and $T_2(\theta_2, Y_2)$. We say that $T_1(\theta_1, Y_1)$ is tighter than $T_2(\theta_2, Y_2)$ if it is closer to $\norm{A}$. To formalize this relationship, we compare the mean absolute error (MAE) of these statistics, that is, we say that $T$ is tighter than $T_2$, if
\begin{equation}
\label{eq:MAE}
\PE \bigl|T(\theta_1, Y_1) - \norm{A}\bigr| < \PE \bigl|T_2(\theta_2, Y_2) - \norm{A}\bigr|\eqsp.
\end{equation}
Of course, one can use another error functions, for example, mean-squared error (MSE), or other ones, depending on the particular applications. Additionally, we can evaluate the tightness of the upper bounds $g(\theta)$ and $\hat{g}(\theta)$ by computing the probabilities
\[
\Prob(T(\theta_1, Y_1) \leq \norm{A})\eqsp, \quad \Prob(T_2(\theta_2, Y_2) \leq \norm{A})\eqsp,
\]
where $\theta_1,\theta_2$ are defined by \eqref{eq:theta_1_2_realizable_def}. Thus, the estimator is allowed to overestimate \(\|A\|_2\), and the main objective is to make this overestimation as small as possible under a fixed failure probability and a small matvec budget. We provide steps of the algorithm with corresponding choice of $\theta$ in \Cref{algo:operator_norm_estimator}.
\par 
In the experiments, we report both quantities: the empirical
underestimation probability, which checks how conservative the theoretical bound is,
and MAE, which measures how much the statistic overestimates the true spectral norm after tuning to the same target \(\delta\).

\section{Numerical experiments}
\label{sec:numerics}

\subsection{Evaluation Setup}
\label{sec:synthetic_data}

We evaluate $\counter(\theta, X)$ against two baselines and one sequential Krylov reference:
\begin{itemize}
    \item \textbf{Vanilla (Gaussian probe) estimator:}
    \[
    \vanilla(\theta, X) = \theta \max_{1 \le i \le 3} \|AX_i\|, \quad 
    \mathbb{P}\bigl(\vanilla(\theta, X) \le \|A\| \bigr) \le \sqrt{\frac{8}{\pi^3}} \, \theta^{-3}.
    \]
    
    \item \textbf{Dixon estimator:}
    \[
    \dixon(\theta, X) = \theta \max \left\{ \sqrt{\|A^\top A X_1\|}, \, \|AX_2\| \right\}, \quad
    \mathbb{P}\bigl(\dixon(\theta, X) \le \|A\| \bigr) \le \frac{2}{\pi} \, \theta^{-3}.
    \]
    \item \textbf{Lanczos-type estimator with 2 sequential matvecs, following \cite{hochstenbach2013probabilistic}}
\end{itemize}

Only $\dixon$ and $\counter$ require $A^\top$ and two sequential matvecs, making them directly comparable.

Synthetic matrices are generated by fixing singular values $(\sigma_1, \ldots, \sigma_k)$, sampling orthogonal $U$, $V$ via QR decomposition of Gaussian matrices, and forming $A = U \Sigma V^\top$.

\subsection{Benchmark Matrices}

We evaluate on the following matrices with varying stable ranks $\rho$:
\begin{itemize}
\item \textbf{Hilbert ($\rho \approx 1.15$)}: $A_{ij} = 1/(i+j)$, $A \in \mathbb{R}^{100 \times 100}$.
\item \textbf{Dominant 0.1 ($\rho = 1.1$)}: $\sigma = (1, \underbrace{0.1, \dots, 0.1}_{10})$, $A \in \mathbb{R}^{100 \times 100}$.
\item \textbf{Dominant 0.5 ($\rho = 3.5$)}: $\sigma = (1, \underbrace{0.5, \dots, 0.5}_{10})$, $A \in \mathbb{R}^{100 \times 100}$.
\item \textbf{Fr\'echet derivative ($\rho \approx 40$)}: $A = D\{\exp H\}$, where $H = -0.01(\Id \otimes T_n + T_n \otimes \Id)$ with tridiagonal $T_n \in \mathbb{R}^{10 \times 10}$, see \cite{kressner2021norm}.
\item \textbf{ResNet Jacobian ($\rho \approx 220$)}: $A \in \mathbb{R}^{784 \times 784}$ is the Jacobian of a convolutional layer (input $28 \times 28$), computed via autodiff \cite{resnet}; see \Cref{sec:description}.

\item \textbf{$A \in \mathbb{R}^{100 \times 100}$ with i.i.d $\mathcal{N}(0, 1)$ entries.}

\item \textbf{Exponential of Roget's Thesaurus semantic graph adjacency matrix} ($\rho \sim 1.004$), $A \in \mathbb{R}^{1022 \times 1022}$
\end{itemize}

We show singular values of the above matrices in Appendix \ref{sec:description}. We fix underestimation probability $\delta = 0.05$ and use the corresponding values of parameter $\theta$. As shown in \Cref{fig:density_comparison}, experiments on both low- and high-rank matrices show that $\counter$ takes lower values. Despite structural differences for ResNet convolution layer and Fr\'echet derivative, both statistics behave similarly. Indeed, when stable rank $\rho$ is large, $\norm{AX}$ with $X \sim \mathcal{N}(0,\Id_n)$ concentrates around $\normf{A}$, and consistently overestimates $\norm{A}$, both for each of the considered estimators. Yet smaller values of $\theta$ used in $\counter(\theta,X)$ implies smaller values of the underlying statistic, see \Cref{fig:density_comparison_F} and \Cref{fig:density_comparison_J}.

\begin{figure}[t]
\centering

\subfigure[Hilbert matrix]{%
  \begin{minipage}[t]{0.48\linewidth}
    \centering
    \includegraphics[width=\linewidth]{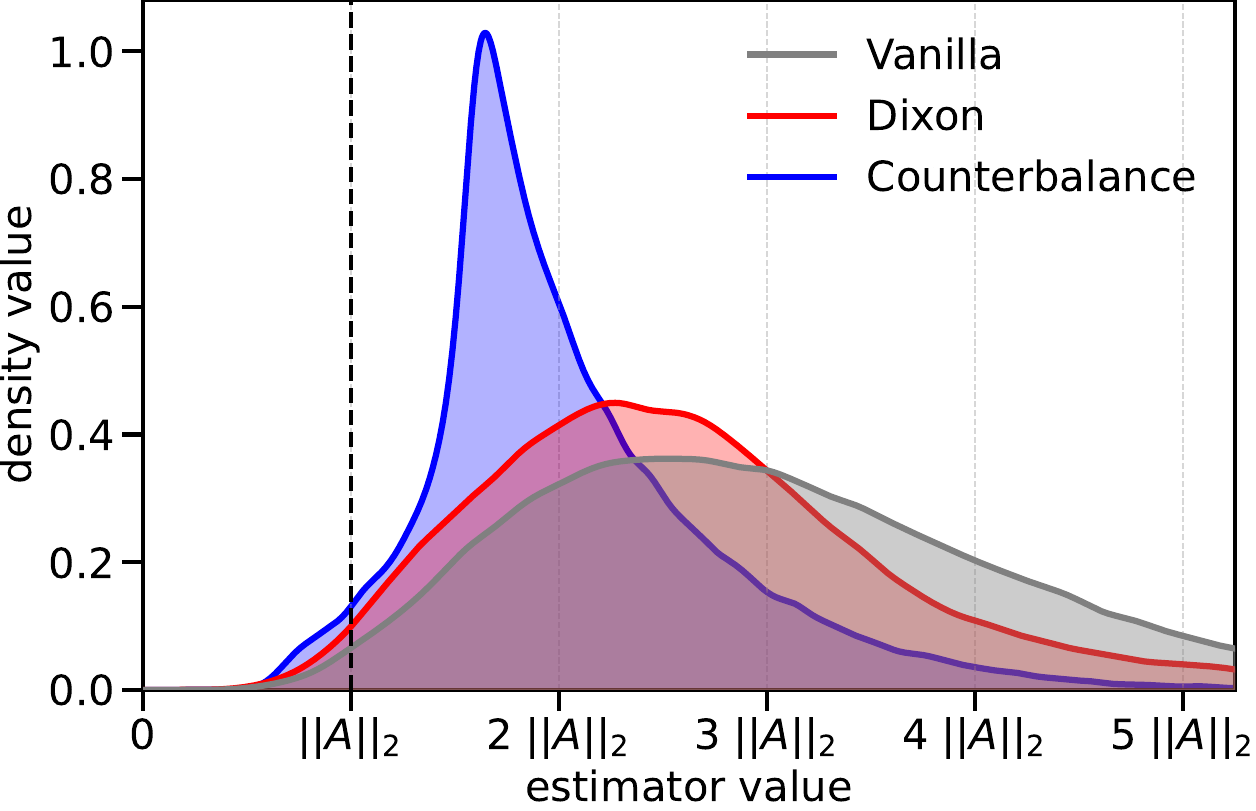}
    \label{fig:density_comparison_H}
  \end{minipage}%
}\hfill
\subfigure[Dominant 0.5]{%
  \begin{minipage}[t]{0.48\linewidth}
    \centering
    \includegraphics[width=\linewidth]{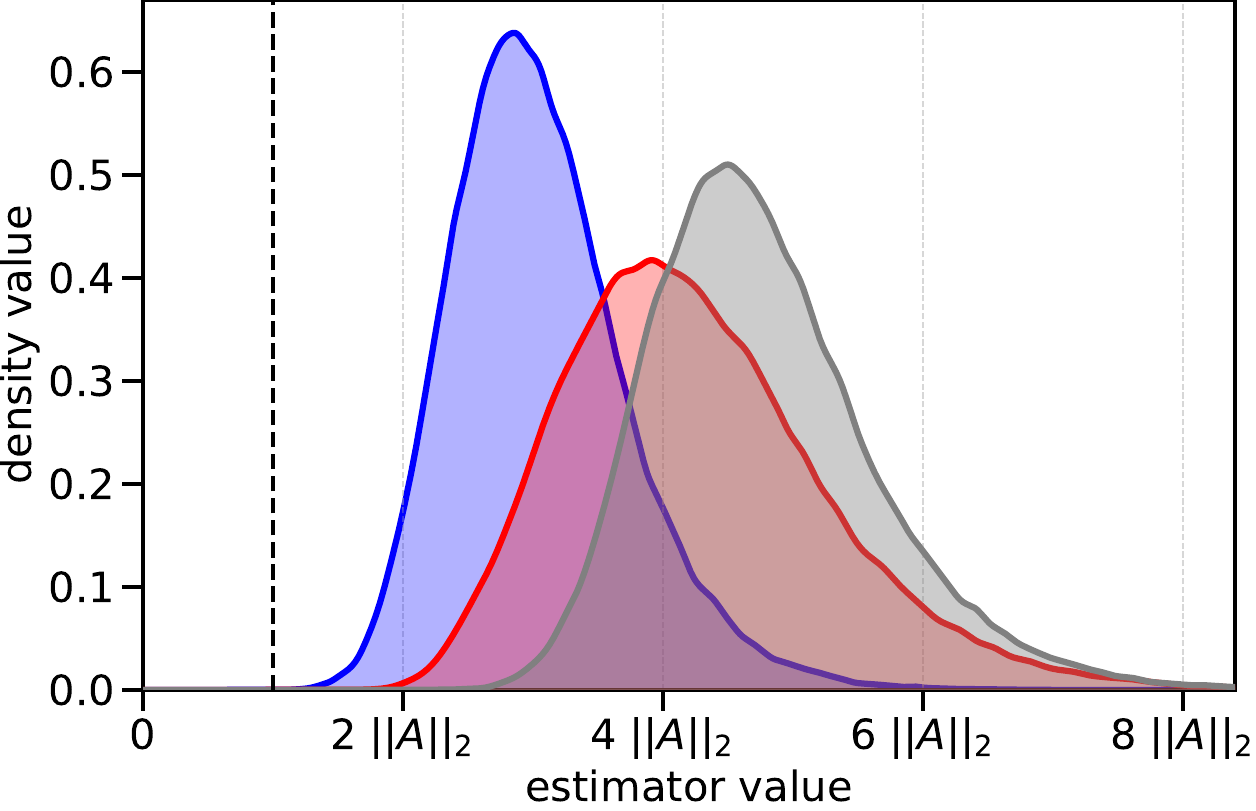}
    \label{fig:density_comparison_dom5}
  \end{minipage}%
}

\vspace{6pt}

\subfigure[Fr\'echet derivative]{%
  \begin{minipage}[t]{0.48\linewidth}
    \centering
    \includegraphics[width=\linewidth]{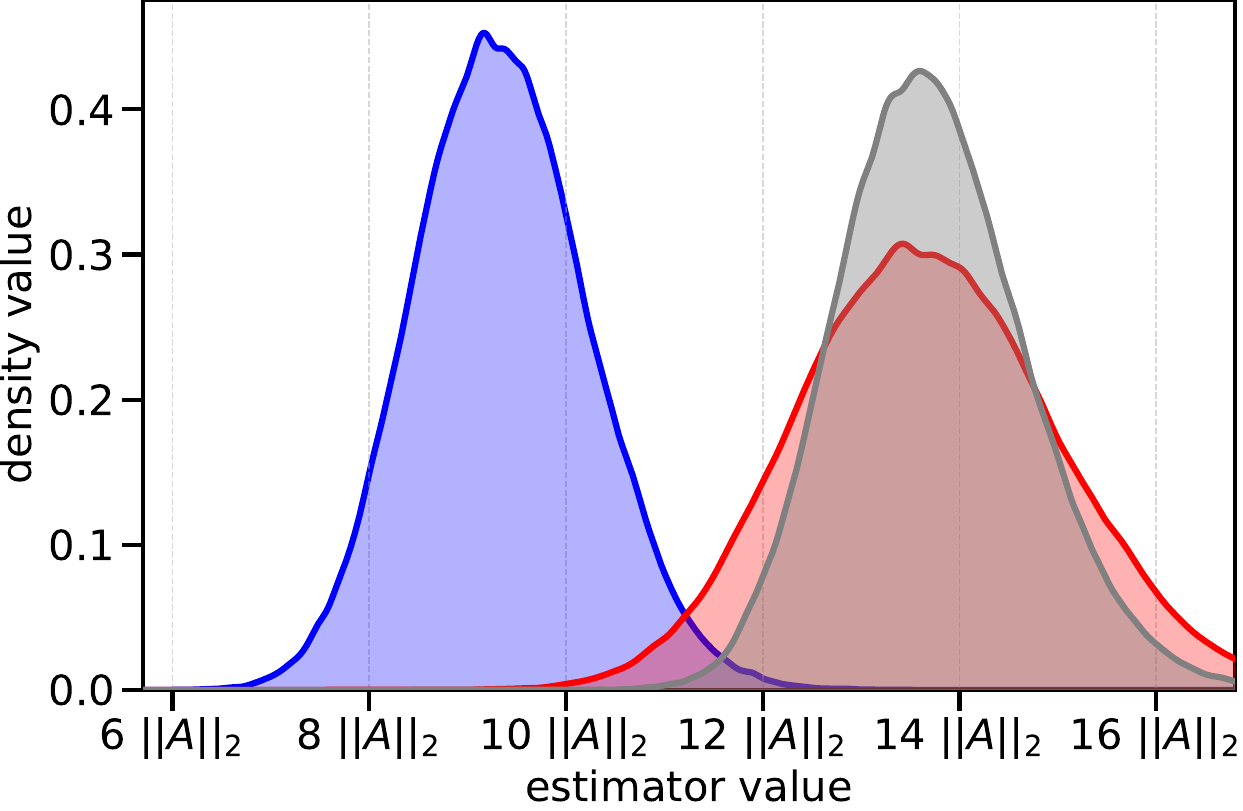}
    \label{fig:density_comparison_F}
  \end{minipage}%
}\hfill
\subfigure[ResNet Jacobian]{%
  \begin{minipage}[t]{0.48\linewidth}
    \centering
    \includegraphics[width=\linewidth]{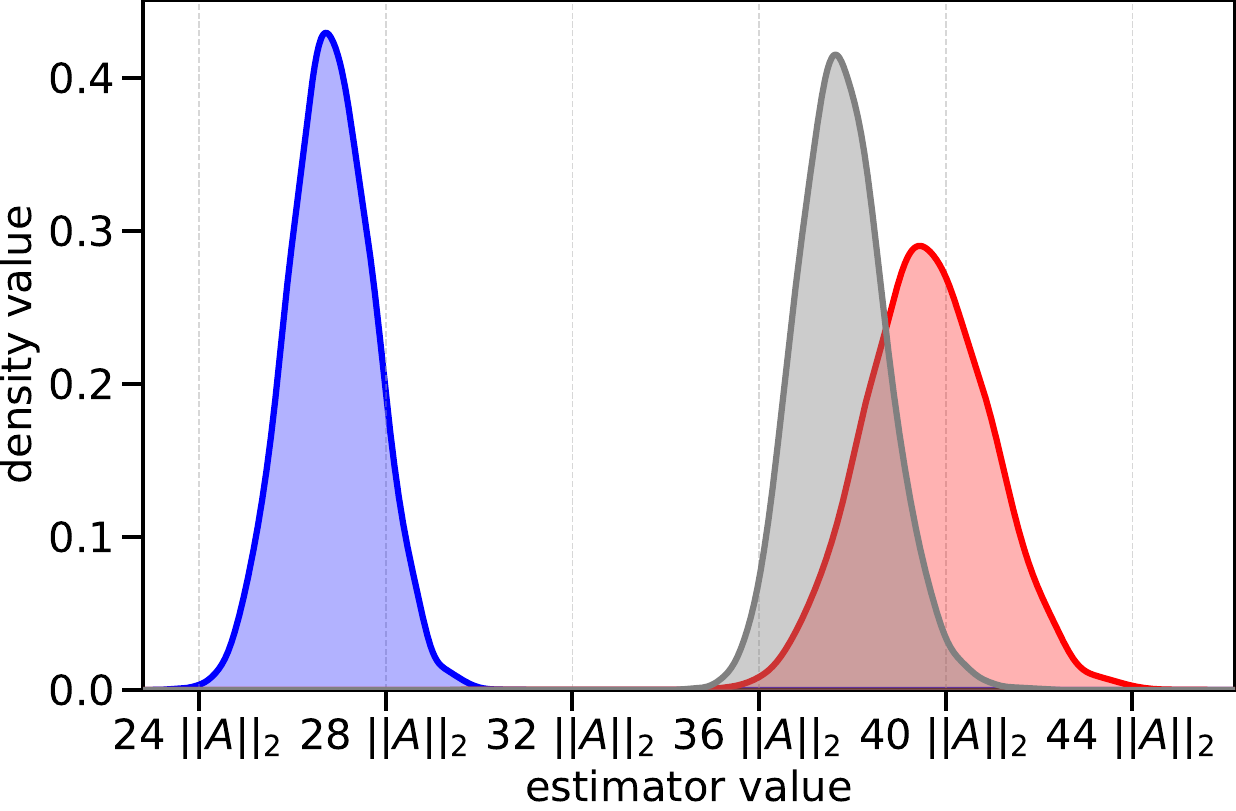}
    \label{fig:density_comparison_J}
  \end{minipage}%
}
\caption{Empirical density comparison for different methods}
\label{fig:density_comparison}
\end{figure}
Comparison with Lanczos-type estimator in \cref{fig:krylov_comparison} illustrates the
complementarity of the two approaches. Counterbalance shows lower $\operatorname{MAE}$
for very low-stable-rank matrices and rapidly decaying spectra, where the
power-type component is already informative. On the other hand, for less
concentrated spectra and for large stable rank, the Krylov reference can produce
tighter bounds, reflecting the additional spectral information accumulated by
Lanczos-type iterations.
\begin{figure}[t]
\centering

\subfigure[Roget's Thesaurus semantic graph adjacency matrix]{%
  \begin{minipage}[t]{0.48\linewidth}
    \centering
    \includegraphics[width=\linewidth]{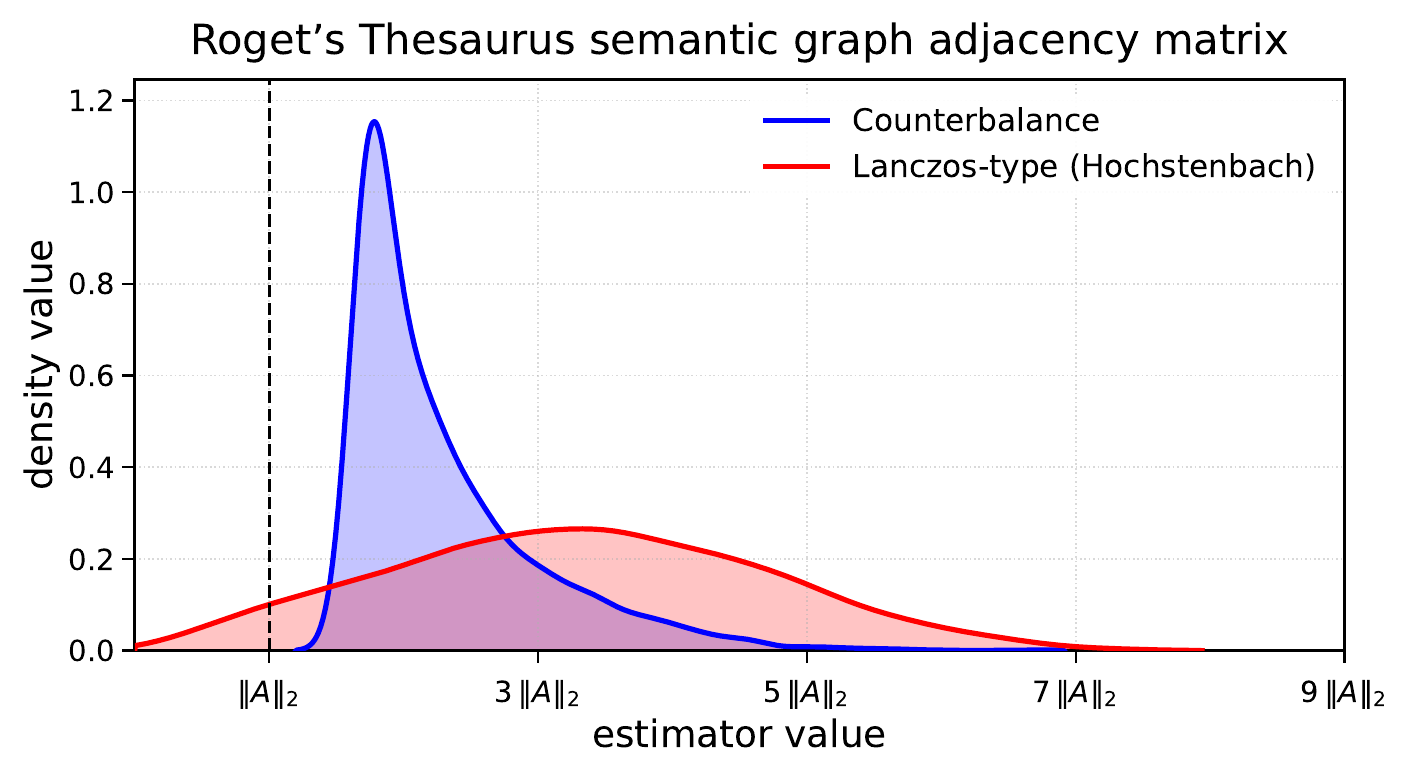}
    \label{fig:krylov_adj}
  \end{minipage}%
}\hfill
\subfigure[Dominant 0.5]{%
  \begin{minipage}[t]{0.48\linewidth}
    \centering
    \includegraphics[width=\linewidth]{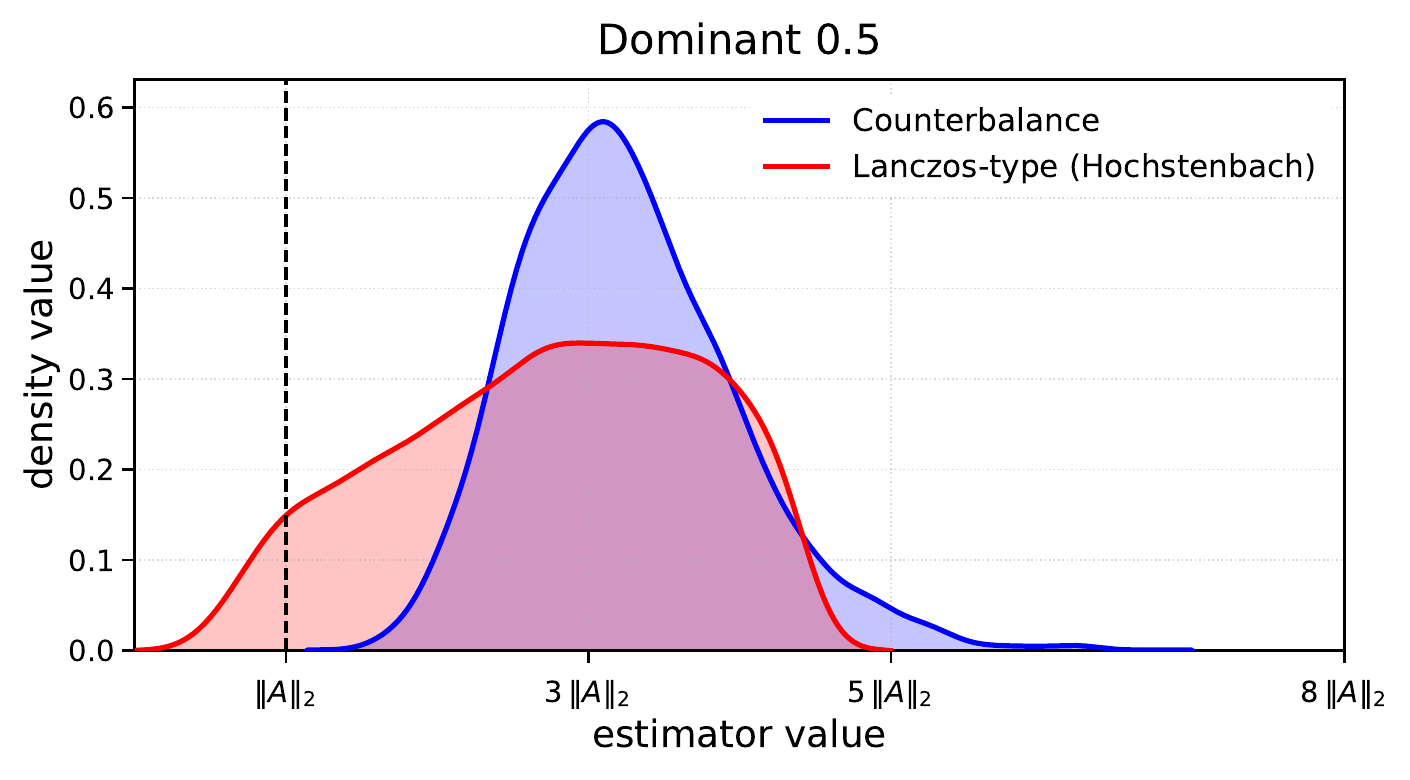}
    \label{fig:krylov_dominant5}
  \end{minipage}%
}

\vspace{6pt}

\subfigure[Standard normal matrix]{%
  \begin{minipage}[t]{0.48\linewidth}
    \centering
    \includegraphics[width=\linewidth]{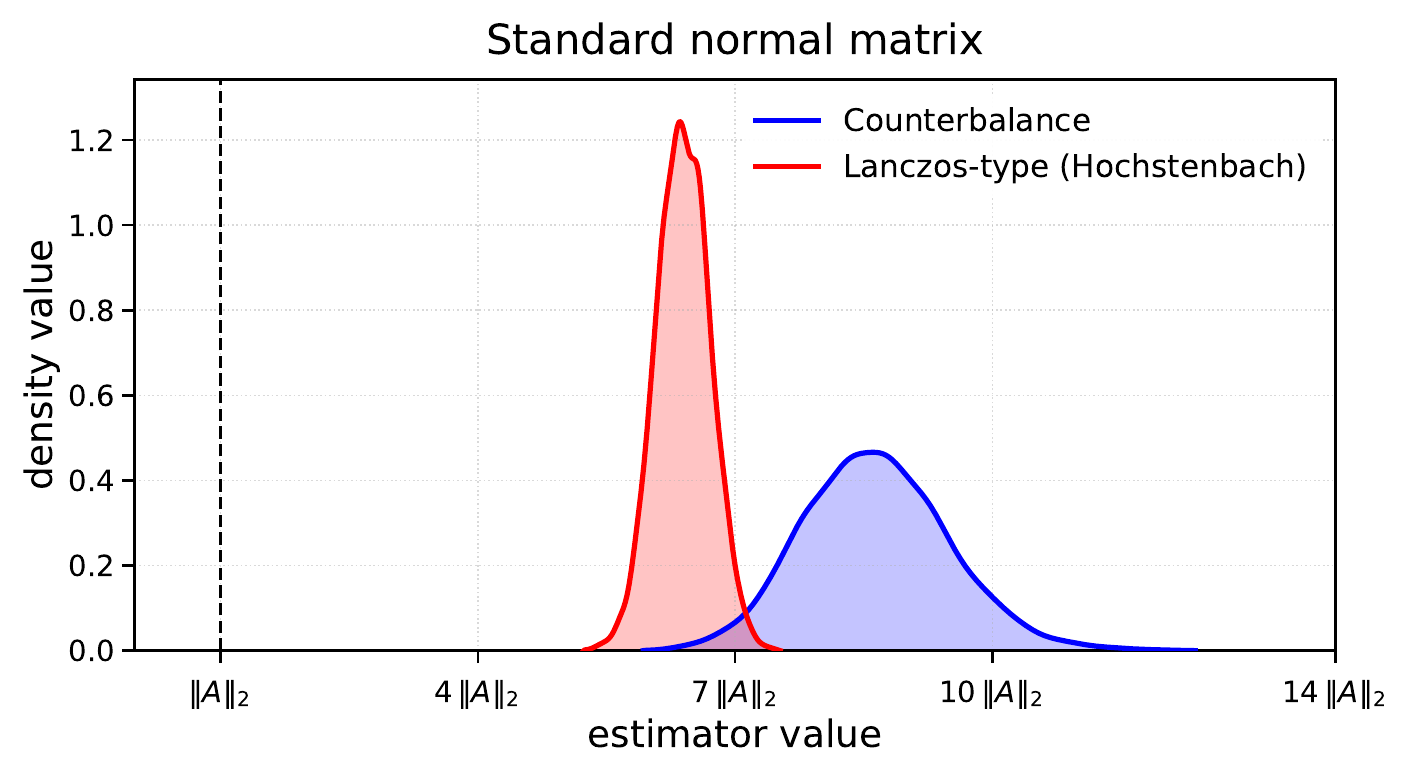}
    \label{fig:krylov_st}
  \end{minipage}%
}\hfill
\subfigure[Hilbert matrix]{%
  \begin{minipage}[t]{0.48\linewidth}
    \centering
    \includegraphics[width=\linewidth]{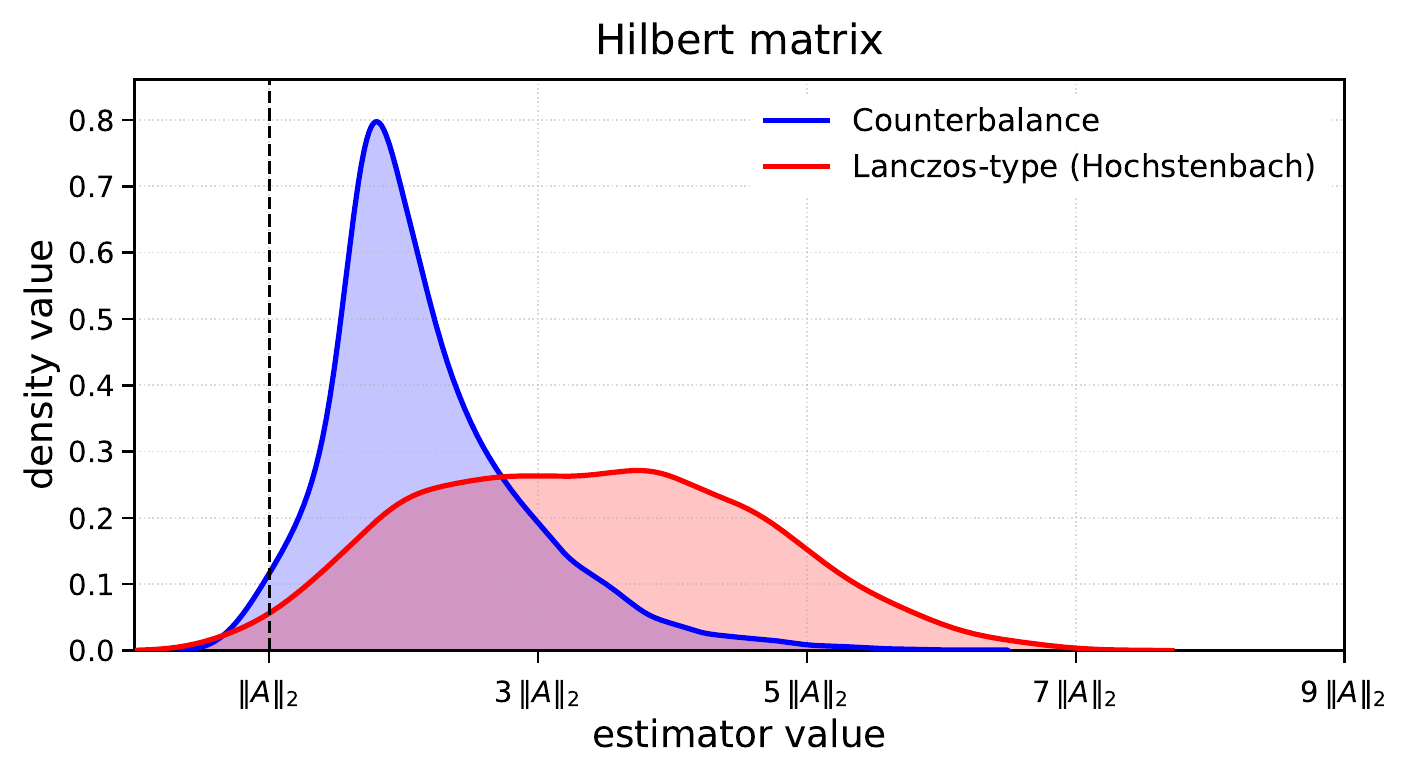}
    \label{fig:krylov_hilbert}
  \end{minipage}%
}
\caption{Empirical density comparison with Lanczos-type estimator.}
\label{fig:krylov_comparison}
\end{figure}

\paragraph{Comparison at increased matvec budget}

\Cref{fig:convergence_plots} illustrates behavior of considered estimators as the number of matvecs increases. In particular, for $k = 3m$ matvecs we obtain
\begin{equation}
    \counter_k(\theta, X) = \max_{1 \le i \le m} \counter(\theta, X_i)
\end{equation}
with corresponding underestimation probability
\begin{equation}
    \prob{\counter_k(\theta, X) \le \norm{A}} \le (g(\theta))^{k/3}
\end{equation}
Since our method uses three matvecs per estimate, we consider the total matvec budget to be a multiplier of three. The curves represent the mean estimate; shaded regions indicate empirical confidence intervals.
We consider $10^6$ realizations of $\vanilla(\theta,X)$, $\dixon(\theta, X)$, and $\counter(\theta,X)$ and observe the line plot describing empirical mean. We report confidence intervals based on 1 standard deviation.
\begin{figure}[t]
\centering

\subfigure[Hilbert matrix.]{
  \begin{minipage}[t]{0.48\linewidth}
    \centering
    \includegraphics[width=\linewidth]{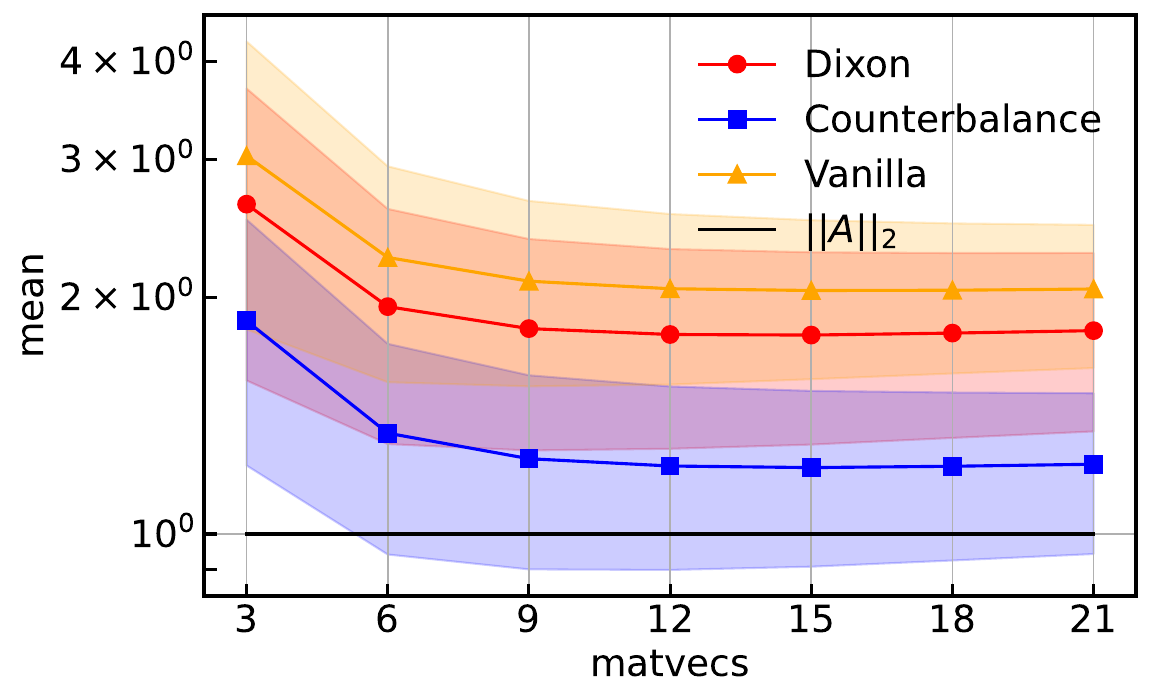}
    \label{fig:plot_H}
  \end{minipage}%
}\hfill
\subfigure[Dominant matrix 0.5.]{
  \begin{minipage}[t]{0.48\linewidth}
    \centering
    \includegraphics[width=\linewidth]{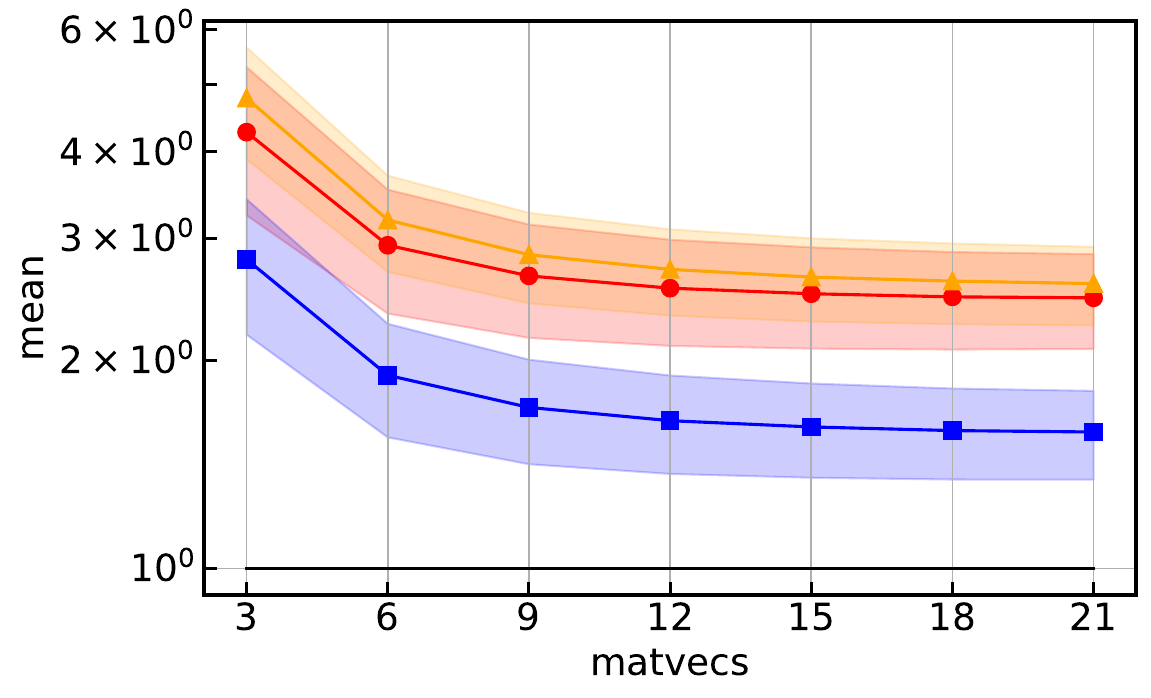}
    \label{fig:plot_D05}
  \end{minipage}%
}

\vspace{6pt}

\subfigure[Fr\'echet derivative.]{
  \begin{minipage}[t]{0.48\linewidth}
    \centering
    \includegraphics[width=\linewidth]{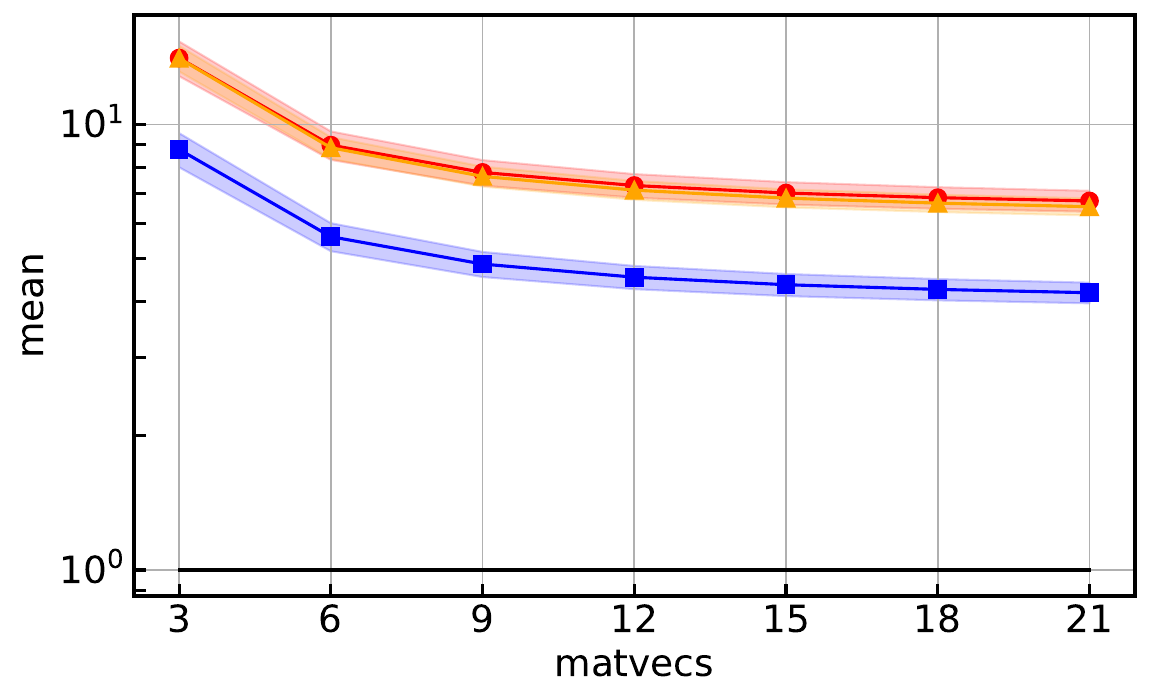}
    \label{fig:plot_F}
  \end{minipage}%
}\hfill
\subfigure[ResNet Jacobian.]{
  \begin{minipage}[t]{0.48\linewidth}
    \centering
    \includegraphics[width=\linewidth]{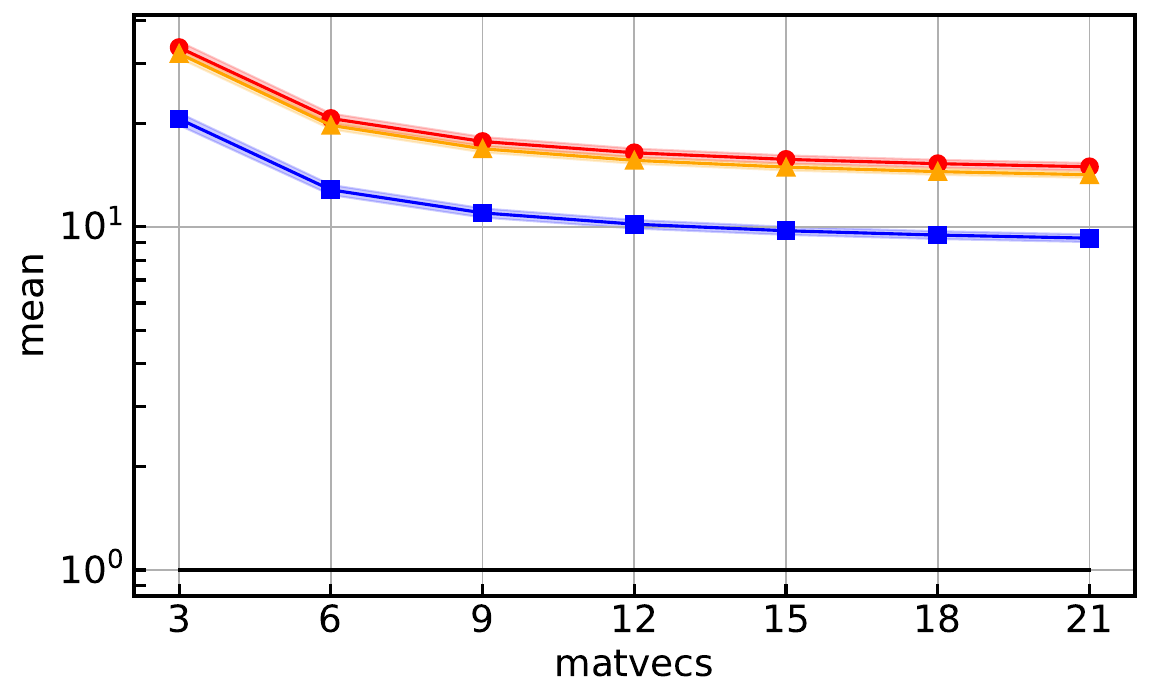}
    \label{fig:plot_J}
  \end{minipage}%
}

\caption{Finite-budget comparison of different estimators as the number of matrix-vector products increases. Solid lines show the mean estimates; shaded areas indicate empirical variance.}
\label{fig:convergence_plots}
\end{figure}

For all matrices presented in \Cref{fig:convergence_plots}, estimators stabilize after 9 matvecs. However, our method consistently converges to smaller values and exhibits lower empirical variance. \Cref{fig:plot_F} and \Cref{fig:plot_J} highlight that all methods achieve low empirical variance on these matrices, since $\norm{AX}$ concentrates around its expectation for matrices with large stable rank $\rho$. 
\par 
\Cref{tab:delta_real} displays $\delta_{real} = \prob{T(\theta, X) \le \norm{A}}$ that measure the tightness of the corresponding bound on the underestimation probability suggested in \eqref{eq:main_bound}. \Cref{tab:delta_real} reports the empirical underestimation probability
\(\delta_{\rm real}=P(T(\theta,X)\le \|A\|_2)\) for the tested statistics at the
nominal level \(\delta=0.05\). For very low-stable-rank matrices, Counterbalance
uses the prescribed failure budget more efficiently than the Vanilla
and Dixon-type baselines, while still respecting the target level. For moderate
and large stable ranks, all considered matrix-agnostic one-shot estimators become
conservative while the Lanczos-type Krylov reference may be closer to
the nominal level.
\begin{table}[H]
\centering
\caption{Values of underestimation probability $\delta_{real} = \prob{T(\theta, X) \le \norm{A}}$ measured for different test statistics with 3 matvecs and $\delta = 0.05$.}
\label{tab:delta_real}
\resizebox{\textwidth}{!}{%
\begin{tabular}{|c|c|c|c|c|c|c|c|}
\hline
$\delta_{\textbf{real}}$ &\textbf{Hilbert} &
\textbf{Rank 2} &
\textbf{Dominant 0.1} & \textbf{Roget's Thesaurus} & \textbf{Dominant 0.5} & \textbf{Fr\'echet derivative} & \textbf{ResNet Jacobian} \\
\hline
  Counterbalance & \textbf{0.034} & \textbf{0.031} & \textbf{0.045} & 0.0001 & 0.0001 & 0 & 0 \\
 Vanilla & 0.011 & 0.019 & 0.016 & 0.013 & 0.0 & 0 & 0 \\
  Dixon & 0.019 & 0.029 & 0.031 & 0.022 & 0.0 & 0 & 0 \\
Lanczos-type & 0.027 & 0.030 & 0.034 & \textbf{0.044} &\textbf{0.04} & 0 & 0 \\
\hline
\end{tabular}%
}
\end{table}

The \Cref{tab:mae} compares the mean absolute error of the considered statistics. All matrices here have spectral norm equals one. Among the
fixed three-matvec randomized estimators, Counterbalance achieves the lowest MAE
on all tested matrices, improving over both the Vanilla and Dixon-type baselines.
The comparison with the Lanczos-type reference is more nuanced: Counterbalance is
better on the very low-stable-rank examples such as Hilbert, Rank~2, Dominant~0.1,
and Roget's Thesaurus, whereas the Lanczos-type reference is tighter on
Dominant~0.5, the Fréchet derivative, and the ResNet Jacobian. This behavior is
consistent with the design of Counterbalance as a fixed low-budget estimator,
rather than a replacement for Krylov subspace methods.

\begin{table}[H]
\centering
\caption{MAE for different test statistics $\delta = 0.05$.}
\label{tab:mae}
\resizebox{\textwidth}{!}{%
\begin{tabular}{|c|c|c|c|c|c|c|c|}
\hline
MAE &\textbf{Hilbert} &
\textbf{Rank 2} &
\textbf{Dominant 0.1} & \textbf{Roget's Thesaurus} &  \textbf{Dominant 0.5} & \textbf{Fr\'echet derivative} & \textbf{ResNet Jacobian}  \\
\hline
 Counterbalance& \textbf{1.01} & \textbf{1.06}& \textbf{0.97} & \textbf{1.28} & 1.99 & 8.65\ & 25.82  \\
 Vanilla& 2.04 & 1.98 & 1.97& 3.35 & 3.77 & 13.46 & 36.78  \\
 Dixon& 1.65 & 1.60 & 1.6 & 2.97 & 3.26 & 13.44 & 38.60  \\
 Lanczos-type & 1.44 & 1.25 & 1.31 & 2.32 & \textbf{1.78} & \textbf{6.95} & \textbf{19.73}  \\
\hline
\end{tabular}%
}

\end{table}

We note that randomized power method itself is loose in comparison to Counterbalance with the worst-case estimate $\rho \le n$, as illustrated in \Cref{fig:power_comparison}.
\begin{figure}[H]
\centering

\subfigure[Fr\'echet derivative, $n=20$]{
  \begin{minipage}[t]{0.48\linewidth}
    \centering
    \includegraphics[width=\linewidth]{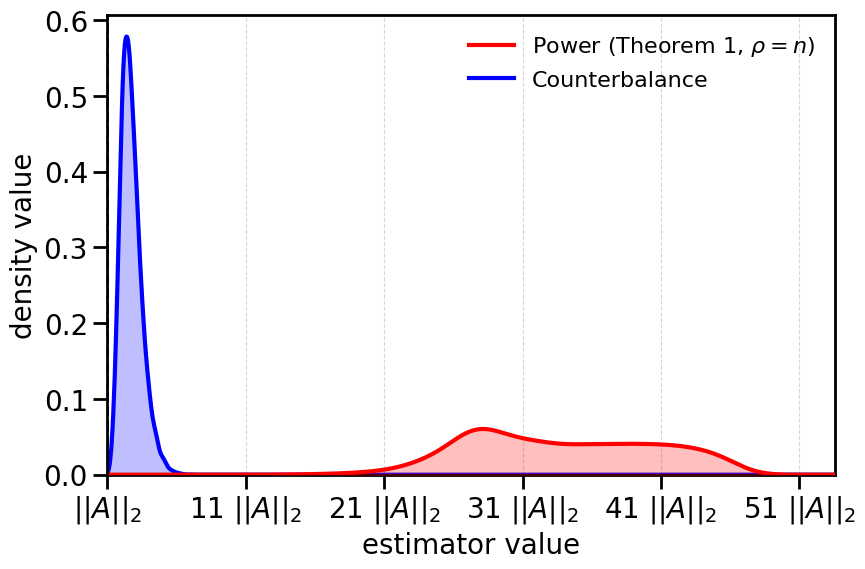}
  \end{minipage}%
}\hfill
\subfigure[Gaussian matrix, $n=100$]{
  \begin{minipage}[t]{0.48\linewidth}
    \centering
    \includegraphics[width=\linewidth]{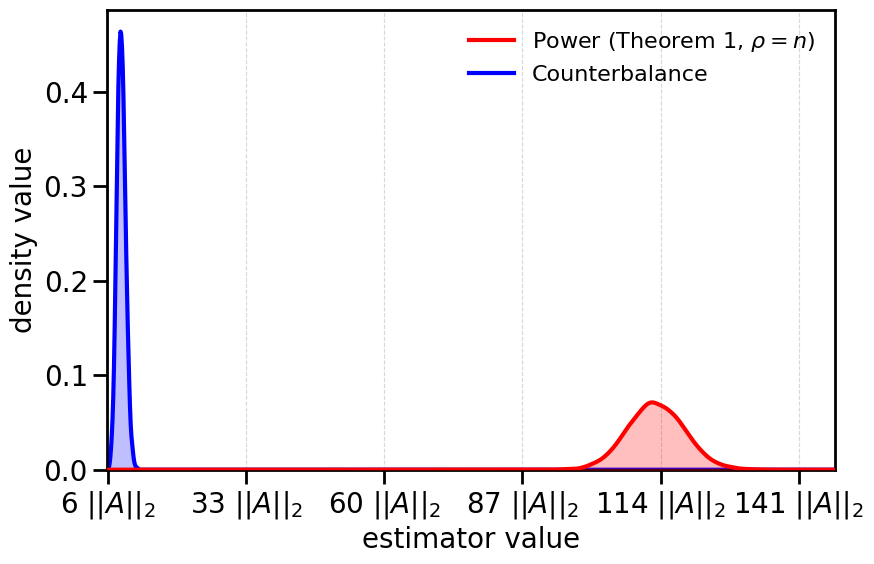}
  \end{minipage}%
}
\caption{Counterbalance and randomized power method comparison}
\label{fig:power_comparison}
\end{figure}

\FloatBarrier
\section{Conclusion}
\label{sec:conclusion}
We introduced Counterbalance, a fixed-budget probabilistic estimator of the matrix spectral norm using three operator products. Its main feature is the combination of a power-ratio term, which is informative for strongly concentrated spectra, and an independent Gaussian term, whose behavior improves as spectral energy spreads.

\par
The principal contribution is the coupled probability analysis used to obtain a calibration independent of unknown spectral information. The stable-rank analysis also identifies an intrinsic limitation of the power-ratio component when only stable rank is retained. Numerical experiments show that Counterbalance
consistently reduces empirical overestimation relative to the considered Vanilla
and Dixon-type baselines at the same nominal failure probability while Lanczos-type bounds can be tighter in high stable rank regimes.

\section{Data availability statement}
The code to reproduce $\theta$ calibration rule  is available  \url{https://github.com/carantire/counterbalance}

\bibliographystyle{unsrtnat}
\bibliography{main_la}

\appendix

\section{Additional experimental details}
\label{sec:description}
We first provide the plots illustrating the decay rates of singular values of non-artificial matrices from \Cref{sec:numerics}. The plot is provided in \Cref{fig:singular_values_decay}.  

\begin{figure}[t]
\centering

\subfigure[Hilbert matrix]{
  \begin{minipage}[t]{0.48\linewidth}
    \centering
    \includegraphics[width=\linewidth]{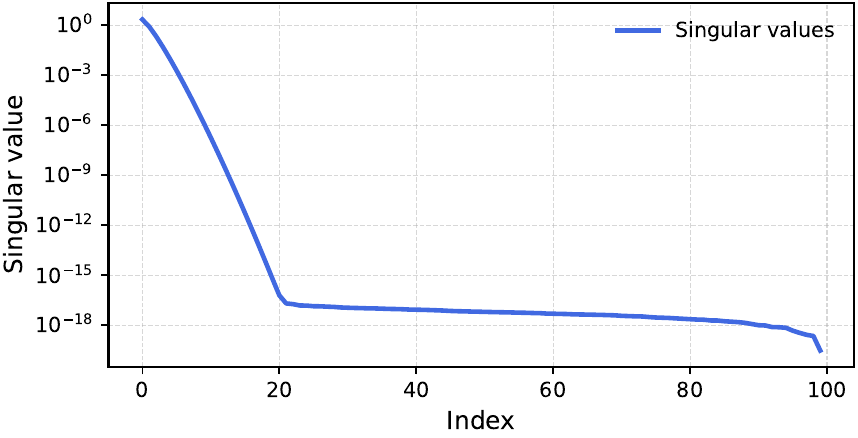}
    \label{fig:sigma_decay_H}
  \end{minipage}%
}\hfill
\subfigure[ResNet Jacobian]{
  \begin{minipage}[t]{0.48\linewidth}
    \centering
    \includegraphics[width=\linewidth]{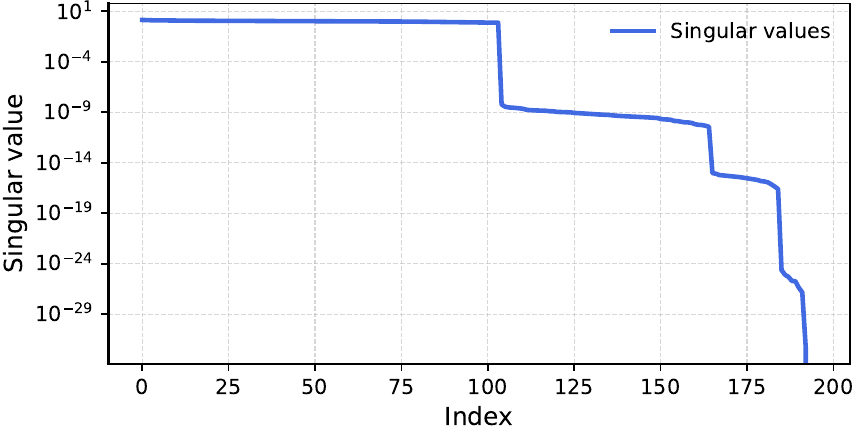}
    \label{fig:sigma_decay_J}
  \end{minipage}%
}

\vspace{6pt}

\subfigure[Fr\'echet derivative]{
  \begin{minipage}[t]{0.48\linewidth}
    \centering
    \includegraphics[width=\linewidth]{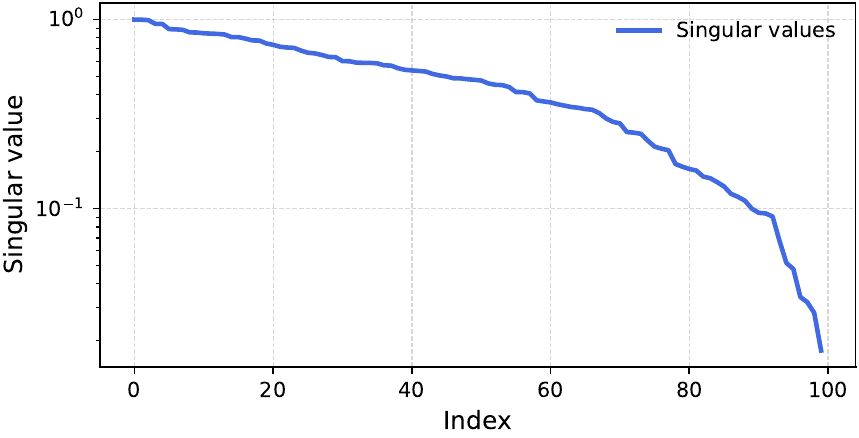}
    \label{fig:sigma_decay_F}
  \end{minipage}%
}

\caption{Singular values for matrices from \protect\Cref{sec:numerics}.}
\label{fig:singular_values_decay}
\end{figure}

For the last matrix from collection in \Cref{sec:numerics} (Dominant 0.1), we provide the density values for considered statistics and scaling of the estimate with increased matvec budget. Plots are provided in \Cref{fig:dominant_01_matrix}. Note that the plots are similar to the ones corresponding to Hilbert matrix.

\begin{figure}[!htbp]
\centering

\subfigure[Scaled number of matvecs]{
  \begin{minipage}[t]{0.48\linewidth}
    \centering
    \includegraphics[width=\linewidth]{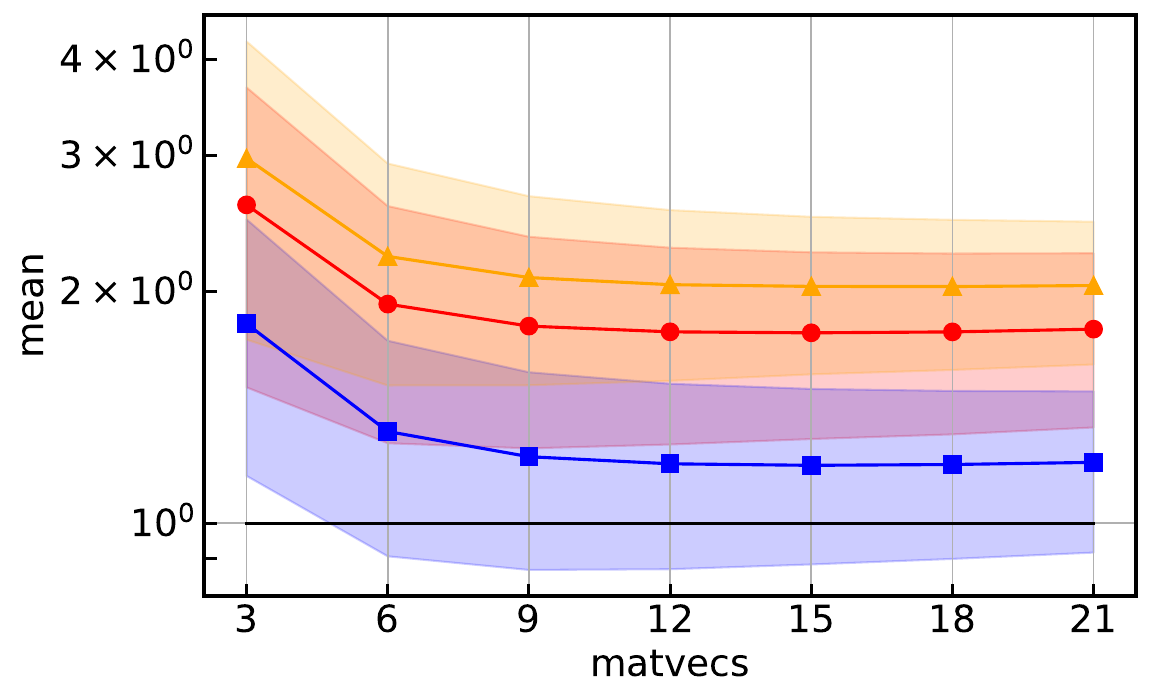}
    \label{fig:plot_dom1}
  \end{minipage}%
}\hfill
\subfigure[Estimator density]{
  \begin{minipage}[t]{0.48\linewidth}
    \centering
    \includegraphics[width=\linewidth]{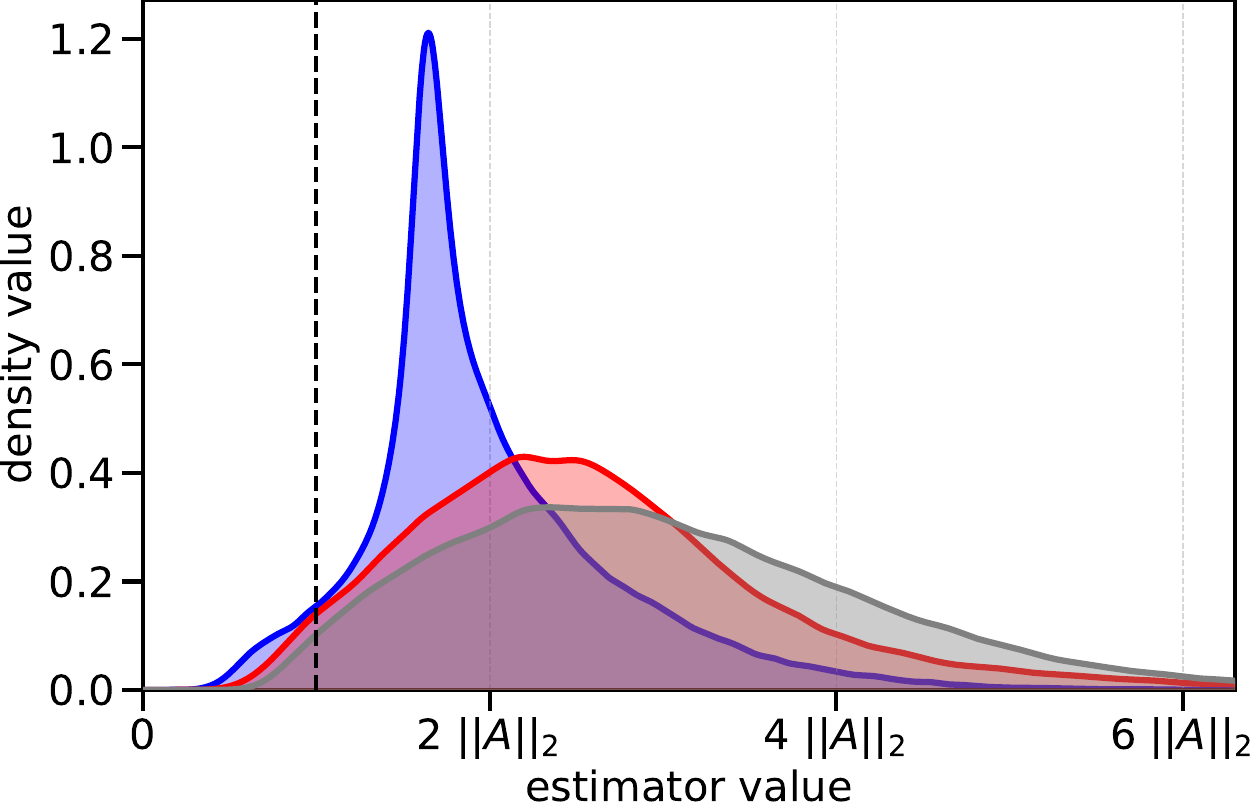}
    \label{fig:density_dom1}
  \end{minipage}%
}

\caption{Results for Dominant 0.1 matrix.}
\label{fig:dominant_01_matrix}
\end{figure}
\FloatBarrier

Then we give detailed description of Fr\'echet derivative experiment considered in \Cref{sec:numerics}.
We use \cite[Section 3.2]{higham2008functions} to obtain that under certain assumptions on the smoothness of \(f\),

\[
f\left( \begin{bmatrix}
A & X \\
0 & A
\end{bmatrix} \right) =
\begin{bmatrix}
f(A) & Df\{A\}(X) \\
0 & f(A)
\end{bmatrix}.
\]
 This implies that we only need to generate sample vector $X \in \mathbb{R}^{10000}$ and reshape it to matrix of size $100 \times 100$, then consider vec$\bracket{Df\{A\}(X)}$ as a result of matvec operation.

        
        
        

\end{document}